\documentclass[10pt]{asme2ej}

\usepackage{graphicx}
\usepackage{color}
\usepackage{amsmath} 
\usepackage{amssymb}
\newtheorem{theorem}{Theorem}
\newtheorem{lem}{Lemma}
\newtheorem{prob}{Problem}
\newtheorem{rem}{Remark}
\usepackage{subcaption}
\usepackage{bm}
\usepackage{url}
\newcommand{\figref}[1]{Fig.\,\ref{#1}}
\newcommand{\Figref}[1]{Figure\,\ref{#1}}
\newcommand{\tabref}[1]{Tab.\,\ref{#1}}
\newcommand{\secref}[1]{Sec.\,\ref{#1}}

\title{
Structural Analysis and Control of a Model of Two-site Electricity and Heat Supply
}

\author{Hikaru Hoshino\thanks{H. Hoshino is currently with Central Research Institute of Electric Power Industry, 1-6-1, Ohtemachi, Chiyodaku, Tokyo, 100-8126, Japan} 
    \affiliation{
	Department of Electrical Engineering,	Kyoto University\\
      Katsura, Nishikyo, Kyoto 615-8510 Japan
    Email: hoshino@dove.kuee.kyoto-u.ac.jp
    }	
}

\author{Yoshihiko Susuki
    \affiliation{ 
	Department of Electrical and Information Systems Osaka Prefecture University\\
	1-1 Gakuencho, Nakaku, Sakai 599-8531 Japan
	Email: susuki@eis.osakafu-u.ac.jp
    }
}

\author{T. John Koo
    \affiliation{ 
	Hong Kong Applied Science and Technology Research Institute\\
	Photonics Centre, 2 Science Park East Ave., Hong Kong Science Park, Shatin, Hong Kong
	Email: johnkoo@astri.org
    }
}

\author{Takashi Hikihara
    \affiliation{ 
	Department of Electrical Engineering Kyoto University\\
      Katsura, Nishikyo, Kyoto 615-8510 Japan
	Email: hikihara.takashi.2n@kyoto-u.ac.jp
    }
}

\begin{document}

\maketitle    

\begin{abstract}
{\it 
This paper 
introduces a control problem of regulation of energy flows in a two-site electricity and heat supply system, where two Combined Heat and Power (CHP) plants are interconnected via electricity and heat flows. 
The control problem is motivated by recent development of fast operation of CHP plants to provide ancillary services of power system on the order of tens of seconds to minutes. 
Due to the physical constraint that the responses of the heat subsystem are not necessary as fast as those of the electric subsystem, the target controlled state is not represented by any isolated equilibrium point, implying that stability of the system is lost in the long-term sense on the order of hours. 
In this paper, we first prove in the context of nonlinear control theory that the state-space model of the two-site system is non-minimum phase due to nonexistence of isolated equilibrium points of the associated zero dynamics. 
Instead, we locate a one-dimensional invariant manifold that represents the target controlled flows completely.  Then, by utilizing a virtual output under which 
the state-space model becomes minimum phase, 
we synthesize a controller that achieves not only the regulation of energy flows in the short-term regime but also stabilization of an equilibrium point in the long-term regime.  Effectiveness of the synthesized controller is established with numerical simulations with a practical set of model parameters. 
}
\end{abstract}

%

\section{Introduction}

This paper introduces a control problem of electricity and heat supply that is expected to be a typical situation in the next-generation energy systems.  
It is indicated in various studies that design of the next-generation energy systems can not be looked as an isolated issue regarding individual energy infrastructures: see e.g. \cite{geidl07:energyhub,omalley13}. 
For example, the so-called Combined Heat and Power (CHP) technology, which exploits waste heat as a by-product of conversion of fuel into electricity, has established an interaction between electric power systems and District Heating and Cooling (DHC) systems \cite{liu16}. 
The above point of view is described as Energy Systems Integration (ESI) \cite{omalley13} and has recently attracted a lot of interest in engineering and science. 
The ESI aims to manage multiple type of energy such as electricity, heat, and natural gas consistently by utilizing interconnections between various energy systems. 

In ESI, design of energy-management architecture is a central object for satisfying specifications of stability, reliability, and efficiency of energy supply. 
In general energy management systems, a high-level planner calculates optimal values of energy flows, and a low-level controller follows the plan while coping with disturbances. 
For a high-level planner, various systematic methods have been proposed for operational optimization: see e.g. \cite{geidl07:opf,chicco09,mancarella14,liu16}. 
While the above studies successfully show the effectiveness of ESI in a steady state, design of a low-level controller is a challenging problem because of the complex physics of energy transfer and conversion in a wide range of scales in space and time \cite{hara13}.  

In this paper,  we consider a minimal prototype of interconnected energy systems, which is termed as \emph{two-site system}, to address a control problem
motivated by recent developments of fast operation of CHP plants. 
Conventionally, a CHP plant has been operated on a slow time scale of heat supply (order of hours) and has not contributed to the fast control of power systems (order of tens of seconds to minutes). 
Recently, several novel operations of a CHP plant have been proposed to provide ancillary services \cite{rebours07_partI} of power systems. 
For example, it is proposed in \cite{galus11} to provide load frequency control using CHP plants and in \cite{shinji08,mueller14} to compensate variable outputs of renewable energy sources. 
For the purpose of controller synthesis, such an operation of CHP plant is naturally formulated as a stabilization problem of an equilibrium point with respect to the model of  electric subsystem. 
However, due to the fact that the responses of the heat subsystem are not necessary as fast as those of the electric subsystem, the state of the heat subsystem is not stabilized to an equilibrium point in the time scale of interest. 
Thus, the desired state of the entire system is not represented by an equilibrium point, and the stability of the entire system is lost in long-term sense (order of hours). 
Instead, we locate a one-dimensional invariant manifold characterizing the target controlled flows of electricity and heat completely. 

The contribution of this paper is twofold. 
First, we provide a structural analysis of the state-space model of the two-site system.  
The system consists of two CHP plants interconnected via electricity and heat flows and its dynamic characteristics are studied in our previous work \cite{nolta14,cdc15,cndpaper16,scipaper17-en}. 
Here, we analyze the input-output property of the model in the context of nonlinear control theory \cite{isidori95,sastry99} and show that the multiscale property of the controlled system implies that the model is a non-minimum phase system due to nonexistence of isolated equilibrium point of the zero dynamics. 
Instead, a one-dimensional invariant manifold is located to characterize the stability property of the zero dynamics. 
Namely, this paper shows that the dynamical analysis of the heat subsystem is naturally extended to the entire interconnected system in the framework of nonlinear control theory. 

Second, a state-feedback controller is synthesized to regulate both electricity and heat flows in the two-site system.  
While the state-space model is a non-minimum phase system under the original output, as shown in our previous work \cite{scipaper17-en}, the model can be dealt with as a minimum phase system under another output.  
Thus, by utilizing an output redefinition method,  we propose a controller indirectly stabilizes the desired invariant manifold to regulated the original output. 
The above output redefinition is motivated by a recently proposed controller for a non-minimum phase model of hypersonic vehicle in \cite{fiorentini12}. 
The output redefinition method in \cite{fiorentini12} converts the original output into a state trajectory of the new zero dynamics and effectively works in the application in this paper. 
Due to the minimum phase property under the redefined output, stabilization of an equilibrium point is also achieved by the common control structure.  
The stabilization of the equilibrium point is necessary when the state of the heat subsystem deviates from an acceptable range due to the long-term dynamics along the invariant manifold.   
Effectiveness of the controller is established by numerical simulations with a practical setting of model parameters.

With the general purpose of designing the next-generation energy systems, various studies have formulated and solved control problems. 
In the context of smart grids, for example, an enhanced automatic generation control is proposed in \cite{ilic12} to address sensing, communication, and control architecture to manage temporal and spacial characteristics of power systems. 
In \cite{dorfler16}, an averaging-based distributed control method is explored to offer the best combination of flexibility and performance without hierarchical decision making and time-scale separations.
In \cite{stegink17}, a problem of maximizing the social welfare is studied while stabilizing both the physical power network as well as the market dynamics.  
As for heat supply systems, an output regulation problem is addressed in \cite{persis14} for large-scale hydraulic networks found in  DHC systems. 
These studies mainly focus on synthesis of distributed controllers, and the control problems addressed in these studies are naturally formulated as stabilization problems of equilibrium points. 
This paper is complementary to these and considers stabilization of an invariant manifold, thereby enabling direct regulation of energy flows in the interconnected energy system exhibiting multiscale dynamics.

The rest of this paper is organized as follows:
In \secref{sec:modeling}, we derive a state-space model of the two-site system and formulate the control problem to regulate electricity and heat flows. 
\secref{sec:analysis} conducts a structural analysis of the derived model. 
In \secref{sec:control}, we present the controller synthesis based on redefinition of the output. 
In \secref{sec:simulation}, the effectiveness of the controller is established with numerical simulation.  
Conclusions of this paper are provided in \secref{sec:conclusion} with future work.

\section{Modeling and problem formulation} \label{sec:modeling}

\begin{figure}[t!]
\centering
\includegraphics[width=0.6\hsize]{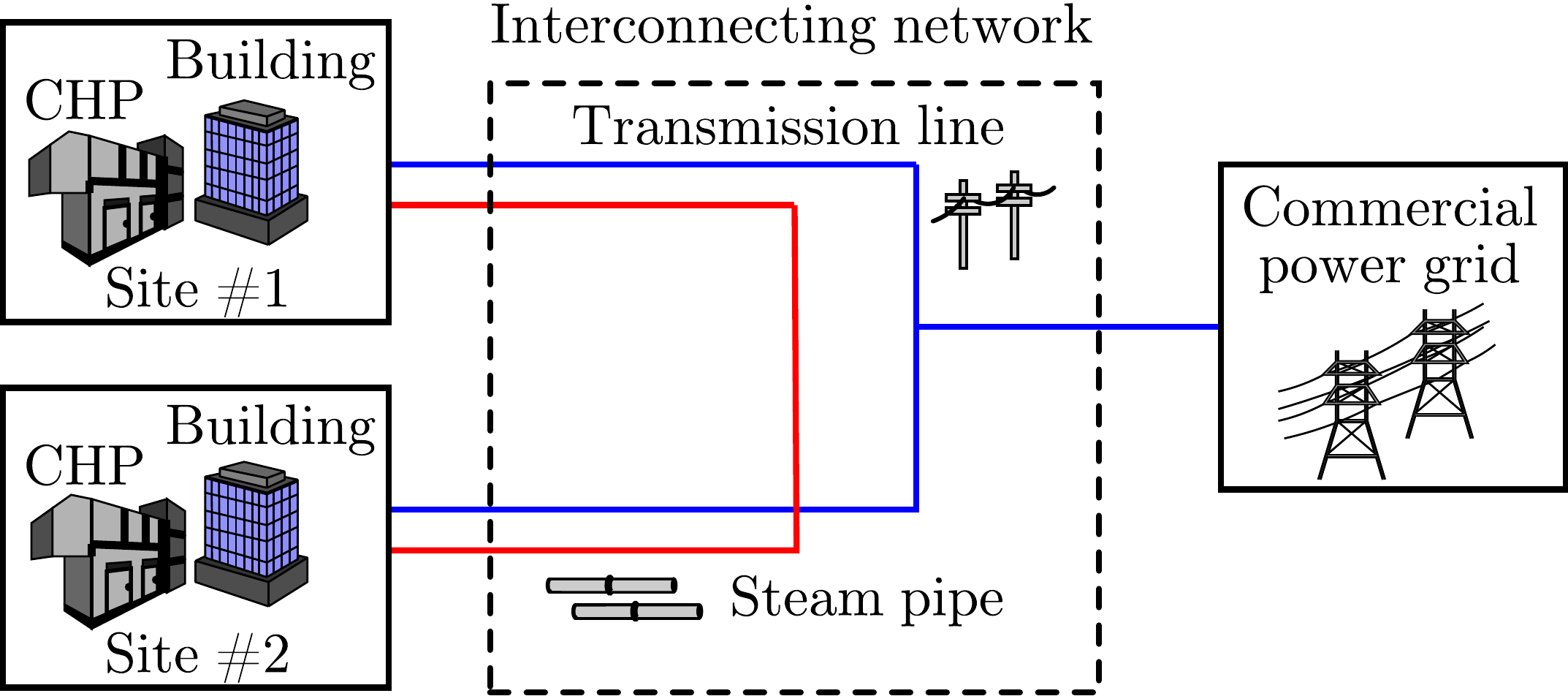}
\caption{Schematic diagram of two-site system for electricity and heat supply} \label{fig:twosite_system}
\end{figure}

This section introduces a two-site system for electricity and heat supply and formulates the output regulation problem studied in this paper. 
%
%
%
%
Figure~\ref{fig:twosite_system} shows the schematic diagram of the two-site system. 
The concept of \emph{site} stands for a unit of energy system that consists of CHP plant, electric load, and heat load \cite{nolta14}. 
The two sites are connected to a commercial power grid through a transmission (distribution) line and it is possible to sell excess electricity to the grid. 
The two sites are also interconnected by a steam pipe.
The practical example of site is a commercial, civil, or large residential building with own CHP plant, and the entire system can be regarded as a minimal prototype of practical DHC systems with CHP plants as reported in e.g. \cite{buoro14,liu16,nedo_kobe18}. 
In the context of electricity supply, the system can be regarded as an extension of the so-called three-node network, for which static and dynamic characteristics have been studied in e.g.  \cite{araposthatis81,ueda92}.
In terms of the heat supply, the system is minimal for considering the heat transfer between different sites in urban areas \cite{nedo_kobe18}.
In this paper, we introduce two subsystems based on their physical characteristics: \emph{electric subsystem} and \emph{heat subsystem}.  
\Figref{fig:energy_flow} shows the diagram of energy flows in the two-site system. 
The CHP plant at each site comprises gas turbine, synchronous generator, and heat recovery boiler. 
At each site $\#i$ for $i=1, 2$, the fuel flow $P_{{\rm g}i}$ to the gas turbine can be controlled via a fuel valve. 
As a result of fuel combustion, the mechanical power $P_{{\rm m}i}$ is produced and transmitted to the generator in each CHP plant.
The generated power $P_{{\rm e}i}$ is then supplied to the electric loads, as well as the commercial power grid. 
The grid is modeled by an infinite bus \cite{kundur94}, which is a voltage source with constant amplitude, frequency, and phase. 
The heat flow $Q'_{{\rm a}i}$ is absorbed by the heat recovery boiler, and $Q'_{{\rm b}i}$ is supplied to the heat loads.
Here, the \emph{electric subsystem} comprises the generators, electric loads, transmission line, and infinite bus in \figref{fig:energy_flow}. 
The \emph{heat subsystem} comprises the heat recovery boilers, heat loads, and pipeline. 
The two subsystems are interconnected via the gas turbines in the CHP plants.

\begin{figure}[!t]
 \centering
 \includegraphics[width=0.6\hsize]{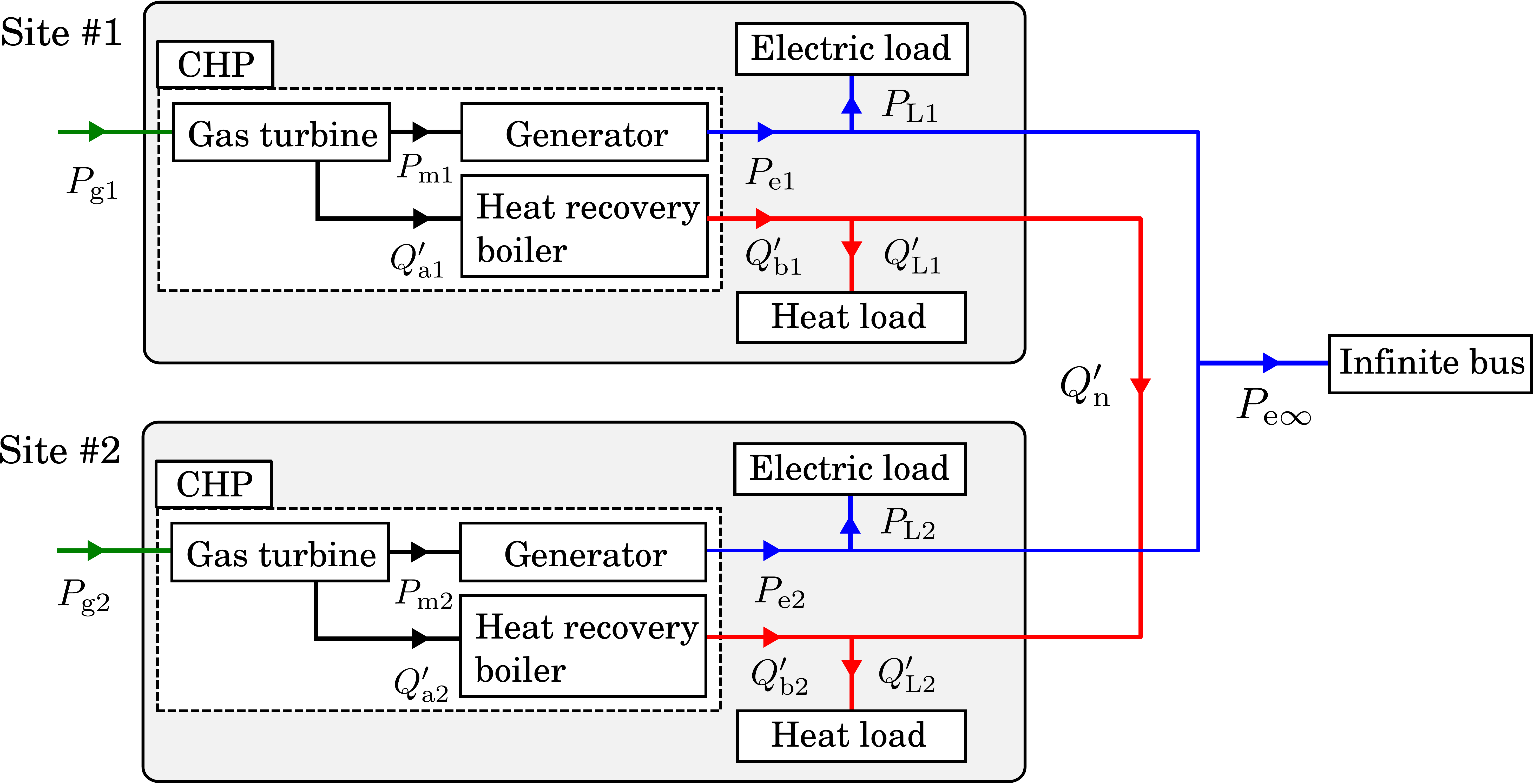}
 \caption{%
 Energy flow diagram of the two-site system. 
 The arrows show the positive directions of the energy flows.}
 \label{fig:energy_flow}
\end{figure}


%
The state variable of the state-space model is given by  $x:=(x_{\rm g}^\top,\,x_{\rm e}^\top,\,x_{\rm h}^\top)^\top \in X \subset \mathbb{R}^{13}$ with the variable $x_{\rm g}$ of the gas turbines, $x_{\rm e}$ of the electric subsystem, and $x_{\rm h}$ of the heat subsystem as follows:
\begin{subequations}
\begin{align}
 & x_{\rm g} = [x_{\rm g1},\, \dots x_{\rm g6}]^\top \in X_{\rm g} \subset \mathbb{R}^{6}, \\
 & x_{\rm e} = [x_{\rm e1},\, \dots x_{\rm e4}]^\top \in X_{\rm e} \subset \mathbb{R}^{4}, \\
 & x_{\rm h} = [x_{\rm h1},\, \dots x_{\rm h3}]^\top \in X_{\rm h} \subset \mathbb{R}^{3}. 
\end{align}
\end{subequations}
The physical meanings of the variables and their domains $X$, $X_{\rm g}$,  $X_{\rm e}$, and  $X_{\rm h}$ are described in Appendix~\ref{sec:model_equations}. 
Then, the state-space model studied in this paper is given as follows:  
\begin{subequations} \label{eq:system_equation}
\begin{align}
 \underbrace{
 \begin{bmatrix} \dot{x}_{\rm g} \\ \dot{x}_{\rm e} \\ \dot{x}_{\rm h} \end{bmatrix}
 }_{\dot{x}}
 & =
 \underbrace{
    \begin{bmatrix} 
       f_{\rm g}(x_{\rm g}) \\ f_{\rm e}(x_{\rm e},\,x_{\rm g}) \\ f_{\rm h}(x_{\rm h},\,x_{\rm g}) 
    \end{bmatrix}
}_{f(x)}
+  
\sum_{i=1}^2
\underbrace{
   \begin{bmatrix} g_{{\rm g}i}(x_{\rm g}) \\ 0  \\ 0  \end{bmatrix} 
}_{g_i(x)}
  u_i, 
 \\
 y & =\begin{bmatrix} y_1& y_2\end{bmatrix}^\top
	 = 
\underbrace{
   \begin{bmatrix} h_{\rm e}(x) & h_{\rm h}(x) \end{bmatrix}^\top
}_{h^\top(x)}
\end{align}
\end{subequations}
where $\dot{x}$ stands for the time derivative of $x$, and $\top$ for the transpose of a vector or matrix. 
The functions  $f_{\rm g}$, $f_{\rm e}$, $f_{\rm h}$, $g_{\rm g1}$, and $g_{\rm g2}$ represent dynamic characteristics of the model and given in in Appendix~\ref{sec:model_equations}. 
The functions $h_{\rm e}$ and $h_{\rm h}$ are the outputs of the electric and heat subsystems, respectively. 


%
The control objective of this paper is to regulate the electric power $P_{\rm e\infty}$ to the infinite bus and the
heat flow rate $Q'_{12}$ between the two sites to given reference values.    
In terms of the electricity supply, $P_{\rm e\infty}$ is an important quantity for providing ancillary services \cite{galus11}.   
The heat flow rate $Q'_{12}$ is involved in the load sharing in the multi-boiler systems \cite{bujak09}
and discussed in \cite{chinese08,buoro14}. 
Thus, the output regulation problem is formulated with respect to the following two outputs $y_1$ and $y_2$:  
\begin{subequations}
\label{eq:output}
\begin{align}
 y_1&= h_{\rm e}(x) := P_{\rm e\infty}(x_{\rm e1}, x_{\rm e3}), \\
 y_2&= h_{\rm h}(x) := Q'_{12}(x_{\rm h2}). 
\end{align}
\end{subequations}

\begin{prob} \label{prob:tracking}
Consider the model \eqref{eq:system_equation} with the outputs in \eqref{eq:output}. 
For given references $Y_1^{\rm ref}$ and $Y_2^{\rm ref}$, find a state-feedback control law such that 
\begin{align}
 \lim_{t\to\infty} (y_1(t),\, y_2(t)) = (Y_1^{\rm ref},\, Y_2^{\rm ref}). 
\end{align}
\end{prob}

Here we briefly reviews a preliminary result on physical and system theoretic properties of the controlled system. 
In particular, for the dynamics of the heat subsystem, the following lemma holds: 
\begin{lem}
\label{thm:NHIM}
 Consider the heat subsystem with constant inputs $u_1^\ast$ and $u_2^\ast$: 
\begin{align}
 \dot{x}_{\rm h} = f_{\rm h}(x_{\rm h}, \, x_{\rm g}^\ast) \label{eq:heat_subsystem_constraints}
\end{align}
where $x_{\rm g}^\ast$ stands for the steady state of the gas turbines under the inputs  $u_1^\ast$ and $u_2^\ast$ satisfying $f_{\rm g}(x_{\rm g}^\ast) + g_{\rm g1}(x_{\rm g}^\ast)u_1^\ast + g_{\rm g2}(x_{\rm g}^\ast)u_2^\ast=0$.
The set(manifold) given by 
\begin{align} \label{eq:invariant_manifold_heat}
 \mathcal{I}_{\rm h} = \left\{ x_{\rm h} \in X_{\rm h} ~\big|~ f_{\rm h1}(x_{\rm h}, x_{\rm g}^\ast) = f_{\rm h2}(x_{\rm h},\,x_{\rm g}^\ast) =0 ~  \right\} 
\end{align}
is invariant under the flow described by the system \eqref{eq:heat_subsystem_constraints}. 
\end{lem}

\begin{proof}
The proof of this lemma is straightforward from \cite{cdc15,cndpaper16} where modeling and analysis of heat supply are performed for the general $n$-site system. Thus it is omitted. \qed
\end{proof}

The lemma follows from the fact that the dynamics of the heat subsystem are independent of the variable $x_{\rm h3}$ standing for the (weighted) averaged pressure level of the two boilers as can be observed in Appendix~\ref{sec:model_equations}. 
As will be shown in  \secref{sec:analysis}, the above result is extended to the entire interconnected energy system, and a similar invariant manifold is identified in the zero dynamics of the model \eqref{eq:system_equation}. 
It will be shown that 
the regulation of the energy flows to arbitrary chosen $Y_1^{\rm ref}$ and $Y_2^{\rm ref}$ is not necessary achievable at a steady state,  
and the desired energy flows are not represented by any equilibrium point of the model \eqref{eq:system_equation}. 


\section{Structural Analysis} \label{sec:analysis}

This section performs a structural analysis of the state-space model \eqref{eq:system_equation}  in the context of nonlinear control theory. 
We prove that the model is a non-minimum phase system \cite{isidori95,sastry99} under the output given in \eqref{eq:output}, and we locate a low-dimensional invariant manifold in the associated zero dynamics to exploit inherent structure of the controlled system.

\subsection{Analysis of non-minimum phase property} \label{sec:analysis_non_minimum_phase}

Here, we analyze zero dynamics of the model \eqref{eq:system_equation} and investigate its non-minimum phase property.
In general, zero dynamics are derived by considering the internal dynamics with the outputs kept to zero for all time. 
This is achieved by applying input-output linearization \cite{isidori95,sastry99}, and the result is summarized in the following lemma:
\begin{lem} \label{thm:normal_form}
 Consider the state-space model \eqref{eq:system_equation} with the outputs in \eqref{eq:output}. 
 Then, there exists an open set $D$ of the state space $X$ such that the model has vector relative degree $\{5,4\}$ at $x \in D$, and the coordinate transformation given by 
\begin{align} \label{eq:coordinate_transformation}
 \Phi :  \mathbb{R}^{13} \ni x \mapsto (\xi_{\rm e},\,\xi_{\rm h},\,\eta) \in  \mathbb{R}^{5} \times  \mathbb{R}^{4} \times  \mathbb{R}^{4} 
\end{align}
with $\xi_{\rm e} = [\xi_{\rm e1},\,\dots,\,\xi_{\rm e5}]^\top$, $\xi_{\rm h} = [\xi_{\rm h1},\,\dots,\,\xi_{\rm h4}]^\top$, and $\eta  =[\eta_{\rm 1},\,\dots,\,\eta_{\rm 4}]^\top$ given by
\begin{subequations}
\begin{align} 
 \xi_{\rm e} & := [ h_{\rm e}(x), \, L_f h_{\rm e}(x) ,\, \dots, \, L_f^4h_{\rm e}(x) ]^\top,  \label{eq:xi_e} \\
 \xi_{\rm h} & := [h_{\rm h}(x), \, L_f h_{\rm h}(x), \dots \,,  L_f^3h_{\rm h}(x) ]^\top,  \\
\eta &  := \left[ x_{\rm g3}, \, x_{\rm e1}, \, x_{\rm e2}, \,  x_{\rm h3} \right]^\top, 
\end{align} 
\end{subequations}
is a diffeomorphism on $D$, where $L_f {h}$ stands for the Lie derivative of $h$ along $f$. 
\end{lem}
\begin{proof}
 See Appendix~\ref{sec:proof_normal_form}. \qed
\end{proof}

The zero dynamics of the model \eqref{eq:system_equation} are described by the variable $\eta$.  
While the zero dynamics are derived by considering the internal dynamics with the outputs kept to zero, for the current control objective, it is technically important to keep the outputs to nonzero values $Y_1^{\rm ref}$ and $Y_2^{\rm ref}$. 
Strictly speaking, this corresponds to modifying the outputs to the set-point error given by
\begin{align}
\label{eq:set_point_error}
  e_{\rm s} :=  \begin{bmatrix} y_1- Y_1^{\rm ref},~ y_2- Y_2^{\rm ref} \end{bmatrix}^\top.    
\end{align}
With respect to the set-point error \eqref{eq:set_point_error}, the zero dynamics are described as follows: 
\begin{align} \label{eq:zero_dynamics}
 \dot{\eta} = \left. q \left( \xi_{\rm e},\, \xi_{\rm h},\,\eta \right) \right|_{\xi_{\rm e}=[Y_1^{\rm ref},\,0,\,\dots,\,0]^\top,\, \xi_{\rm h}=[Y_2^{\rm ref},\,0,\,0,\,0]^\top}
\end{align}
where the vector-valued function $q$ 
is given by 
\begin{align}
 q(\xi_{\rm e},\, \xi_{\rm h},\, \eta) := (r \circ \Phi^{-1})(\xi_{\rm e},\, \xi_{\rm h},\, \eta) \label{eq:internal_dynamics}
\end{align}
with 
\begin{align} \label{eq:function_r}
 r(\cdot) :=\begin{bmatrix}
 f_{\rm g3}(\cdot) & f_{\rm e1}(\cdot) & f_{\rm e2}(\cdot) & f_{\rm h3}(\cdot)  \end{bmatrix}^\top. 
\end{align}
To investigate the non-minimum phase property, it is crucial that the variable $\eta$ contains the averaged pressure level $\eta_4 := x_{\rm h3}$. 
The following lemma, which states that the variable $\eta_4$ does not affect the internal dynamics, will be instrumental in identifying an invariant manifold in the zero dynamics: 

\begin{lem} \label{lem:independence_eta4}
 The function $q(\xi_{\rm e},\,\xi_{\rm h},\,\eta)$ is independent of the variable $\eta_4$. 
\end{lem}
\begin{proof}
 See Appendix~\ref{sec:independence_eta4}. 
\qed
\end{proof}

Based on the above lemma, the following theorem identifies an invariant manifold characterizing the stability property of the zero dynamics:  

\begin{theorem}
  Consider the zero dynamics of the model \eqref{eq:system_equation} with respect to the set-point error \eqref{eq:set_point_error}.  
  Suppose that the following set $\mathcal{I}$ given by
\begin{align} \label{eq5:invariant_manifold}
 \mathcal{I}=&  \left\{ \eta ~\big|~ F_1(\eta)=F_2(\eta)=F_3(\eta)=0 \right\} 
\end{align}
forms an embedded submanifold of the state space of the zero dynamics, where the functions $F_1$ to $F_3$ are the first three elements of the following vector-valued function:   
\begin{align}
 F (\eta) := q ([Y_1^{\rm ref},\,0,\,\dots,\,0]^\top,\, [Y_2^{\rm ref},\,0,\,0,\,0]^\top,\,\eta ).
\end{align}
Then, the manifold $\mathcal{I}$ is invariant with respect to the flow described by \eqref{eq:zero_dynamics}--\eqref{eq:function_r}. 
 \end{theorem}
\begin{proof}
 The invariance of $\mathcal{I}$ under the zero dynamics is equivalent to $F(p) \in {\rm T}_p \mathcal{I}$ for all $p \in \mathcal{I}$, where ${\rm T}_p \mathcal{I}$ stands for the tangent space of $\mathcal{I}$ at $p \in \mathcal{I}$.  
Here, from Lemma~\ref{lem:independence_eta4}, the tangent space ${\rm T}_p \mathcal{I}$ is given as follows: 
\begin{align}
 T_p \mathcal{I} = {\rm ker} 
 \begin{bmatrix}
 \dfrac{\partial F_1}{\partial \eta_1}(p) & \dfrac{\partial F_1}{\partial \eta_2}(p) & \dfrac{\partial F_1}{\partial \eta_3}(p)  & 0 \\[3mm]
 \dfrac{\partial F_2}{\partial \eta_1}(p) & \dfrac{\partial F_2}{\partial \eta_2}(p) & \dfrac{\partial F_2}{\partial \eta_3}(p)  & 0 \\[3mm]
 \dfrac{\partial F_3}{\partial \eta_1}(p) & \dfrac{\partial F_3}{\partial \eta_2}(p) & \dfrac{\partial F_3}{\partial \eta_3}(p)  & 0 \\ 
 \end{bmatrix} 
\end{align}
where ${\rm ker} M$ stands for the kernel space of a linear mapping described by a finite matrix $M$. 
Furthermore, from the assumption, the vectorfield at $p \in \ \mathcal{I}$ is written as $[0~0~0~F_4(p)]^\top$.  
Thus, it can be verified that $F(p) \in {\rm T}_p \mathcal{I}$.
\qed
\end{proof} 

While the above theorem locates an invariant manifold in the zero dynamics, this implies that the desired state to satisfy Problem~1 can be represented by the following invariant manifold: 
\begin{align}
 \mathcal{I}_{\rm s} = \bigl\{ (\xi_{\rm e},\,\xi_{\rm h},\,\eta) ~\big|~ & \xi_{\rm e}=[Y_1^{\rm ref},\,0,\,\dots,\,0]^\top,\,  \notag \\ &
 \xi_{\rm h}=[Y_2^{\rm ref},\,0,\,0,\,0]^\top, \notag \\ & 
 q_1(\eta)=q_2(\eta)=q_3(\eta)=0 \bigr\}.
\end{align}
In \secref{sec:control} we will synthesize a controller that indirectly stabilizes the above invariant manifold through a suitable output redefinition. 

\begin{rem}
This theorem implies that the model \eqref{eq:system_equation} is a non-minimum phase system because of no isolated asymptotically stable 
equilibrium point in the zero dynamics. 
This holds even if there exists an equilibrium point satisfying $e_s=0$.  
This is because any equilibrium point of the zero dynamics is not asymptotically stable as shown below. 
If there exists an equilibrium point $\eta^\ast$ of the zero dynamics, from Lemma~\ref{lem:independence_eta4}, any point in the following set is also an equilibrium point: 
\begin{align} \label{eq:equilibrium_set}
 \mathcal{I}^\ast=&  \left\{ ( \eta_1^\ast ,\,\eta_2^\ast ,\,\eta_3^\ast, \,  \eta_4^\ast +a )^\top ~\big|~ a \in \mathbb{R} \right\}. 
\end{align}
Thus, no equilibrium point of the zero dynamics is isolated. 
This implies that no equilibrium point is asymptotically stable, and hence the model is a non-minimum phase system. 
\end{rem}

\begin{figure}[t!]
\begin{minipage}{\hsize}
\centering
 \includegraphics[width=0.6\hsize]{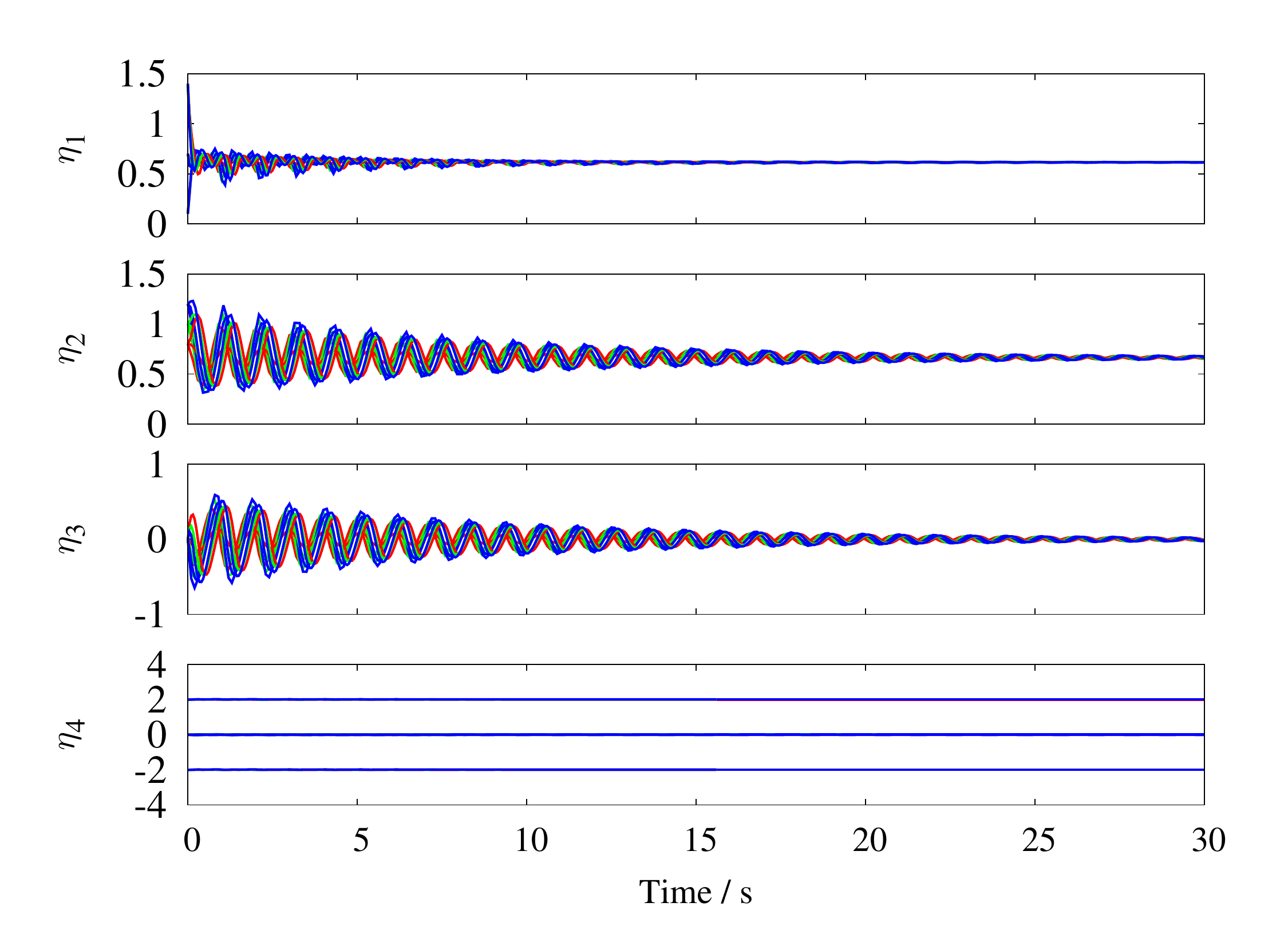}
\subcaption{Responses of the variables $\eta_1,\,\eta_2,\,\eta_3,\,\eta_4$} \label{fig5:zero_dynamics_ep_response}
\vspace{-10mm} 
\hspace{-10mm}
\includegraphics[width=0.6\hsize]{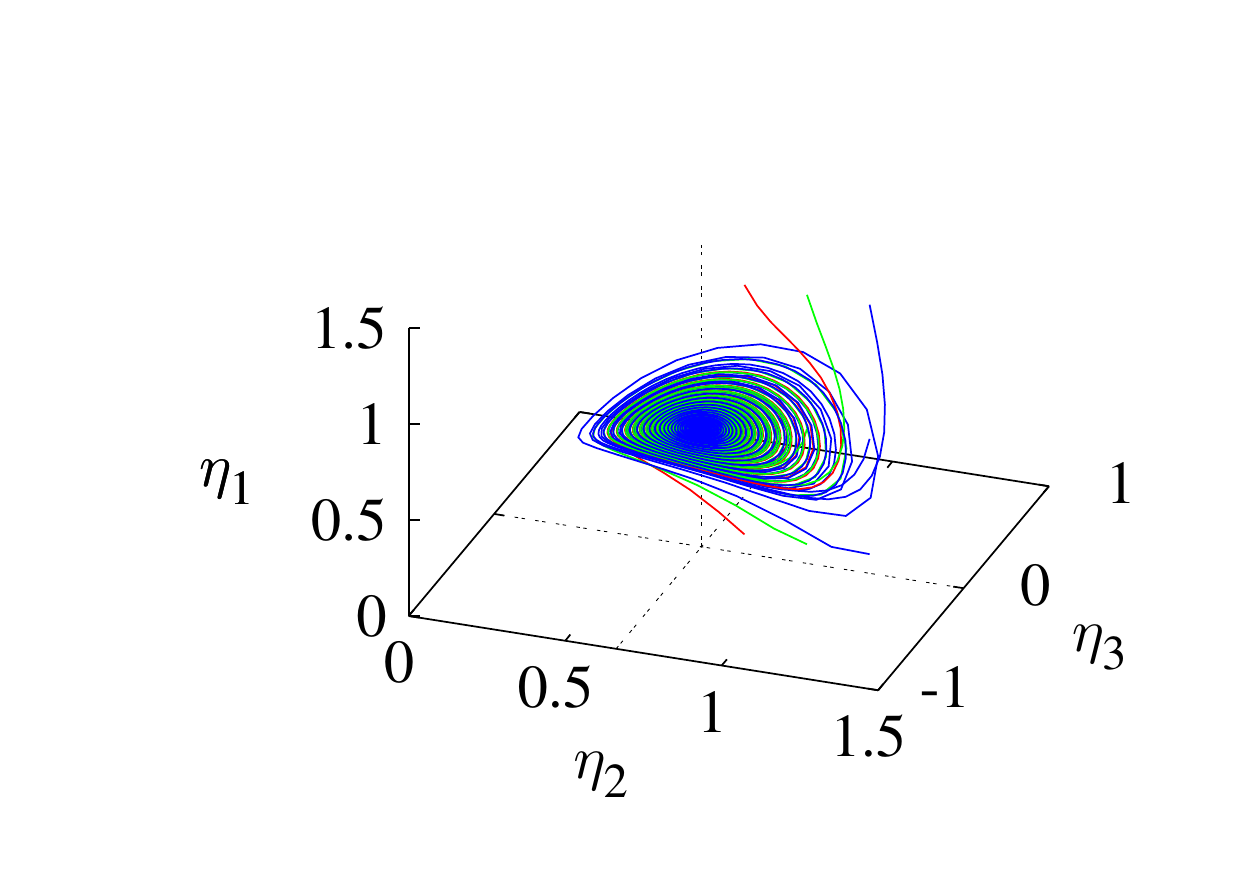}
\subcaption{Projection to $(\eta_1,\,\eta_2,\,\eta_3)$ space} \label{fig5:zero_dynamics_ep_eta123}
\vspace{-10mm} 
\hspace{-5mm}
\includegraphics[width=0.6\hsize]{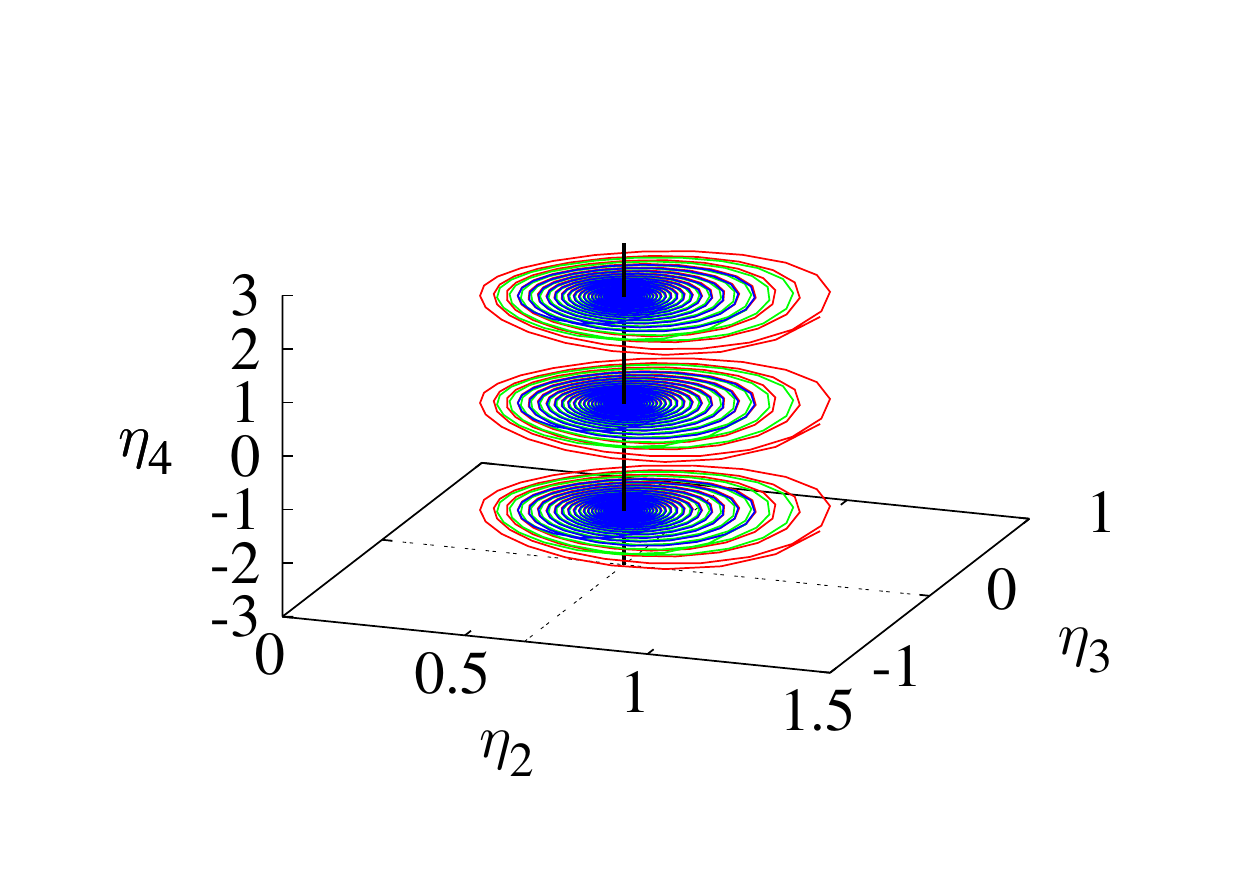}\
\subcaption{Projection to $(\eta_2,\,\eta_3,\,\eta_4)$ space} \label{fig5:zero_dynamics_ep_eta234}
\end{minipage}
\caption{Trajectories of the zero dynamics described by \eqref{eq:zero_dynamics} with a setting of $Y_1^{\rm ref}$ and $Y_2^{\rm ref}$ such that there exists a set of equilibrium points: $Y_1^{\rm ref}=1.0$ and $Y_2^{\rm ref} =1.69$.} \label{fig5:zero_dynamics_ep}
\end{figure}

\begin{figure}[t!]
\begin{minipage}{\hsize}
\centering
 \includegraphics[width=0.6\hsize]{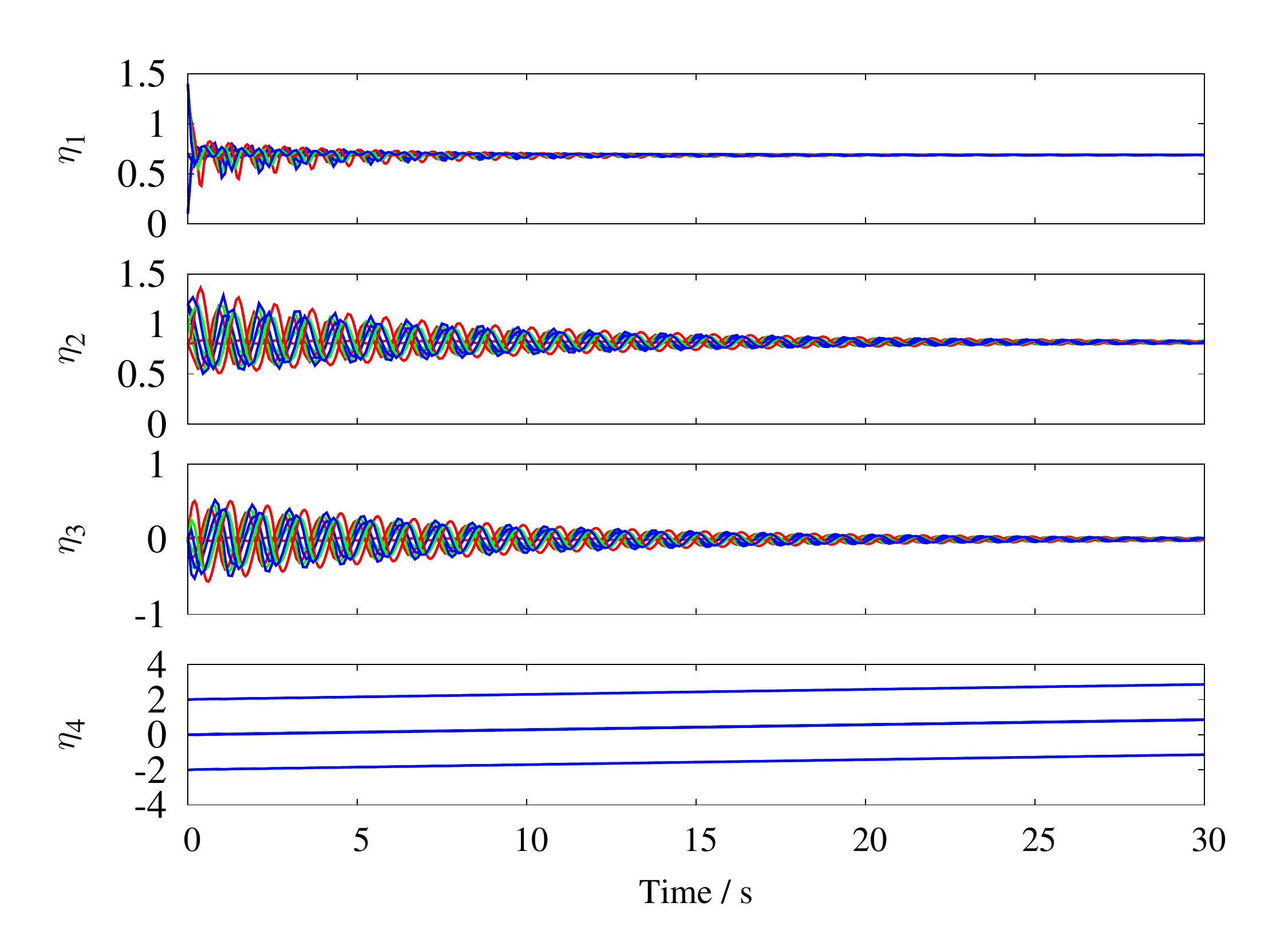}
\subcaption{Responses of the variables $\eta_1,\,\eta_2,\,\eta_3,\,\eta_4$} 
\vspace{-10mm} 
\hspace{-10mm}
\includegraphics[width=0.6\hsize]{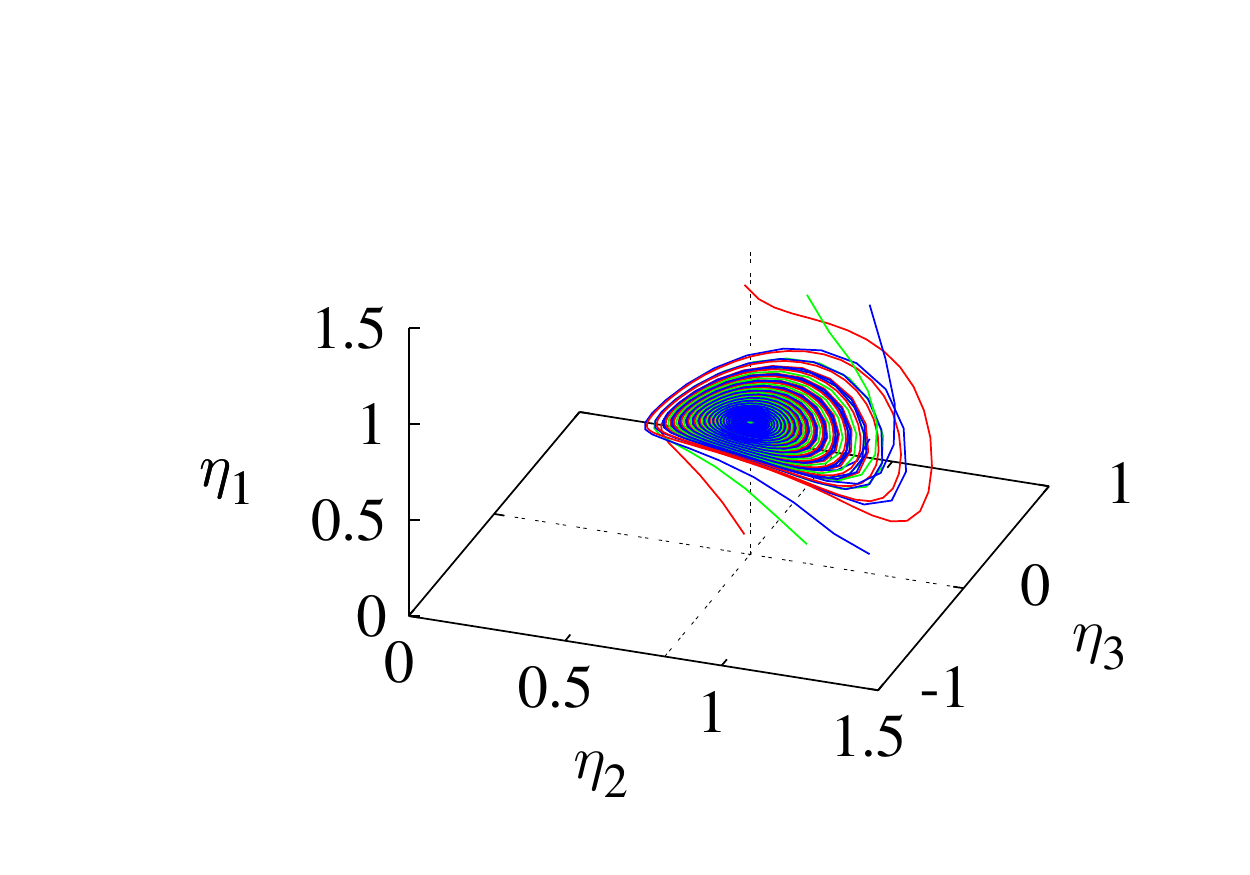}
\subcaption{Projection to $(\eta_1,\,\eta_2,\,\eta_3)$ space}
\vspace{-10mm} 
\hspace{-5mm}
\includegraphics[width=0.6\hsize]{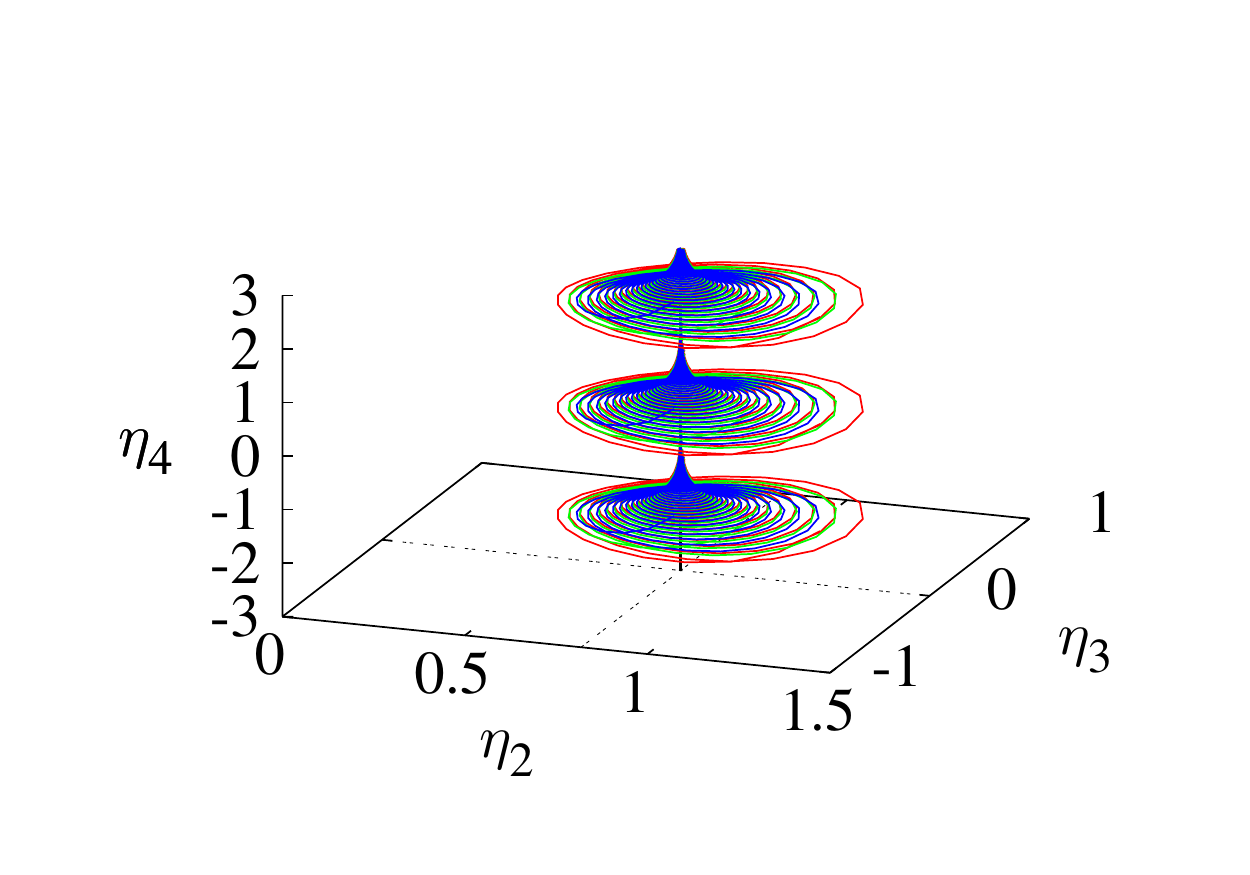}
\subcaption{Projection to $(\eta_1,\,\eta_2,\,\eta_3)$ space}
\end{minipage}
\caption{%
Trajectories of the zero dynamics described by \eqref{eq:zero_dynamics} with a setting of $Y_1^{\rm ref}$ and $Y_2^{\rm ref}$ such that no equilibrium point exists: $Y_1^{\rm ref}=1.2$ and $Y_2^{\rm ref} =1.69$. %
} \label{fig5:zero_dynamics_no_ep}
\end{figure}

\subsection{Numerical simulation of zero dynamics} \label{sec:simulation_zero_dynamics}

This subsection provides numerical simulations of the zero dynamics to illustrate the existence of the invariant manifold.  
The setting of parameters is shown in \tabref{tab:parameters}. 
First, we provide a numerical simulation of the zero dynamics under a case where the references $Y_1^{\rm ref}$ and $Y_2^{\rm ref}$ are chosen so that there exists a set of equilibrium points of the zero dynamics. 
Figure~\ref{fig5:zero_dynamics_ep} shows several trajectories with the following values of $Y_1^{\rm ref}$ and $Y_2^{\rm ref}$: 
\begin{align}
 Y_1^{\rm ref}=1.0, \quad Y_2^{\rm ref} =1.69. 
\end{align}
Figure~\ref{fig5:zero_dynamics_ep_response} shows the responses of the $i$-th element $\eta_i$ under trajectories from fifteen different initial conditions. 
In the figure, the \emph{red} lines show the trajectories from the initial conditions with $\eta_2=0.8$, the \emph{green} lines with $\eta_2=1.0$, and the \emph{blue} lines with $\eta_2=1.2$. 
To illustrate the existence of the invariant manifold mentioned above, Figs.~\ref{fig5:zero_dynamics_ep_eta123} and \ref{fig5:zero_dynamics_ep_eta234} provide projections of the trajectories in the full four-dimensional phase space to lower-dimensional sub-spaces. 
In \figref{fig5:zero_dynamics_ep_eta123}, several trajectories are projected to the $(\eta_1,\,\eta_2,\,\eta_3)$ space, and it is observed that %
these variables converge to constants (see also the first to third rows of \figref{fig5:zero_dynamics_ep_response}).
In \figref{fig5:zero_dynamics_ep_eta234}, where trajectories are projected to the $(\eta_2,\,\eta_3,\,\eta_4)$ space, the black line shows the set of equilibrium points that form the invariant manifold given by \eqref{eq5:invariant_manifold}. 
It is confirmed that all the trajectories converge to the invariant manifold.

Furthermore, a similar invariant manifold can be identified when no equilibrium point exists. 
The invariant manifold is parallel to that in \figref{fig5:zero_dynamics_ep} consisting of equilibrium points.  
Figure~\ref{fig5:zero_dynamics_no_ep} shows the trajectories of the zero dynamics with the following values of $Y_1^{\rm ref}$ and $Y_2^{\rm ref}$: 
\begin{align}
 Y_1^{\rm ref}=1.2, \quad Y_2^{\rm ref} =1.69. 
\end{align} 
As in \figref{fig5:zero_dynamics_ep}, all the trajectories converge to an invariant manifold in the phase space of the zero dynamics. 
However, in the current case, the values of $\eta_4$ slowly increase after the convergence to the invariant manifold, implying that the stability of the system is lost in long-terms sense. 
The variable $\eta_4$ parametrizes the points on the invariant manifold and represents the slow dynamics along the manifold. 

\section{Controller synthesis} \label{sec:control}

This section studies a controller synthesis for Problem~\ref{prob:tracking} through a suitable redefinition of the outputs. 
We propose a common control scheme that achieves not only the regulation of energy flows in the short-term regime but also stabilization of an equilibrium point in the long-term regime. 

\subsection{Main idea}

\begin{figure*}[!t]
\centering 
\includegraphics[width=0.9\hsize]{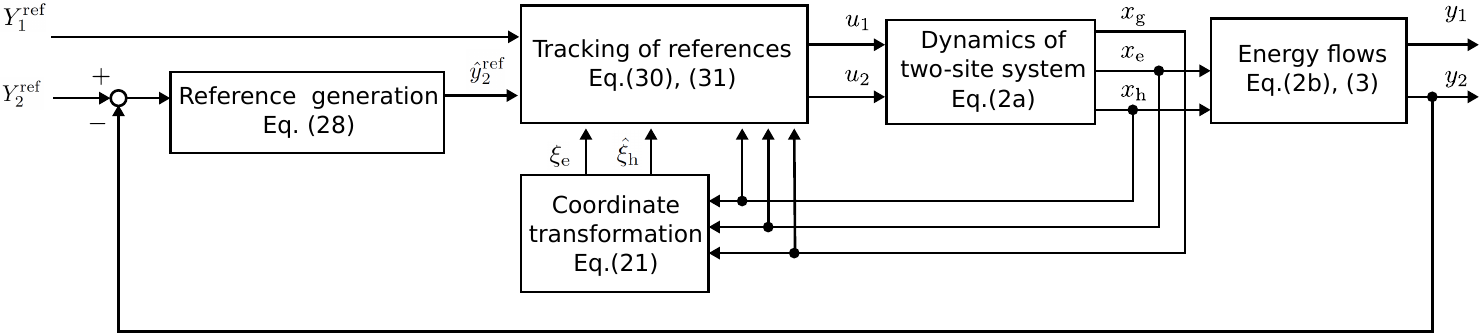}
\caption{Block diagram of the proposed controller for the two-site electricity and heat supply}
\label{fig:control_scheme}
\end{figure*}

The controller synthesized in this paper indirectly stabilizes the desired invariant manifold while dealing with the model as a minimum phase system. 
Due to this, stabilization of an equilibrium point is also possible by the same control scheme when the averaged pressure level $\eta_4$ deviates from an acceptable range due to the dynamics along the invariant manifold.  
The controller is synthesized by a suitable redefinition of the output, which is motivated by the recently proposed output redefinition method for a non-minimum phase model of hypersonic vehicle in \cite{fiorentini12}. 
The method converts the original output into a state trajectory of the new zero dynamics. 
As will be reviewed in \secref{sec:normal_form_redefined}, it is shown in \cite{scipaper17-en} that the model \eqref{eq:system_equation} can be treated as a minimum phase system by considering the following outputs: 
\begin{subequations} \label{eq:redefined_output}
\begin{align}
 &y_1=h_{\rm e}(x) = P_{\rm e\infty}(x_{\rm e1},\,x_{\rm e3}),  \\
 &\hat{y}_2 = \hat{h}_{\rm h}(x):= \eta_4 = x_{\rm h3}. 
\end{align}
\end{subequations}
In this paper, we show that the regulation of energy flows can be achieved with the outputs in \eqref{eq:redefined_output} by considering an augmented system with the integral error. 
Figure~\ref{fig:control_scheme} shows the block diagram of the proposed controller. 
For given references of the outputs $y_1$ and $y_2$, the part of reference generation calculates the reference $\hat{y}_2^{\rm ref}(t)$ for the redefined output $\hat{y}_2$.

\subsection{Normal form under redefined output} \label{sec:normal_form_redefined}

Here, we derive the normal form of the model \eqref{eq:system_equation} with respect to the outputs in \eqref{eq:redefined_output}. 
Towards the controller synthesis in the next subsection, it will be proven that the new internal dynamics are independent of the redefined output $\hat{y}_2$. 
Note that this corresponds to Lemma~\ref{lem:independence_eta4} since the output $\hat{y}_2$ is equivalent to $\eta_4$.  
First, the following lemma summarizes the result of input-output linearization: 
\begin{lem}[\cite{scipaper17-en}] \label{lem:redefined_normal_form}
 Consider the model \eqref{eq:system_equation} with the outputs in \eqref{eq:redefined_output}. 
 Then, there exists an open subset $\hat{D}$ of the state space $X$ such that the model has vector relative degree $\{5, 3\}$ at $x\in \hat{D}$, and the coordinate transformation given by 
\begin{align} \label{eq:coordinate_transformation_redefined}
 \hat{\Phi}(x): \mathbb{R}^{13} \ni x \mapsto (\xi_{\rm e},\, \hat{\xi}_{\rm h},\, \hat{\eta}) \in \mathbb{R}^5 \times \mathbb{R}^3  \times \mathbb{R}^5 
\end{align}
is a diffeomorphism on $\hat{D}$, where $\xi_{\rm e}$ is given by \eqref{eq:xi_e}, and $\hat{\xi}_{\rm h}=[\hat{\xi}_{\rm h1}, \, \hat{\xi}_{\rm h2} \, \hat{\xi}_{\rm h3}]^\top$ and $\hat{\eta} = [\hat{\eta}_1 ,\, \dots, \, \hat{\eta}_5]^\top$ by
\begin{subequations}
\begin{align}
 \hat{\xi}_{\rm h} &  
:= \begin{bmatrix} \hat{h}_{\rm h}(x) & L_f \hat{h}_{\rm h}(x)  & L_f^2\hat{h}_{\rm h}(x)  \end{bmatrix}^\top, \\
 \hat{\eta} & 
:=  \begin{bmatrix}x_{\rm g3},\, x_{\rm e1},\, x_{\rm e2},\, x_{\rm h1}, \, x_{\rm h2} \end{bmatrix}^\top. \label{eq:eta_bar}
\end{align} 
\end{subequations}
\end{lem}
While the above lemma is proven in \cite{scipaper17-en}, for self-consistency of the paper, we provide the proof in Appendix~\ref{sec:proof_redefined_normal_form}.
From Lemma~\ref{lem:redefined_normal_form}, with the outputs in \eqref{eq:redefined_output}, the model \eqref{eq:system_equation} has the following normal form: 
\begin{subequations} \label{eq:normal_form}
\begin{align}
 \dot{\xi}_{\rm e1} &= \xi_{\rm e2}, \\ 
   & \hspace{2mm} \vdots \notag \\
 \dot{\xi}_{\rm e5} &= b_1(\xi_{\rm e},\, \hat{\xi}_{\rm h},\, \hat{\eta}) + \sum_{i=1}^2 a_{1i}(\xi_{\rm e},\, \hat{\xi}_{\rm h},\, \hat{\eta}) u_i, \\ 
 \dot{\hat{\xi}}_{\rm h1} &= \hat{\xi}_{\rm h2}, \\ 
  & \hspace{2mm} \vdots \notag \\
 \dot{\hat{\xi}}_{\rm h3} &= b_2(\xi_{\rm e},\, \hat{\xi}_{\rm h},\, \hat{\eta}) + \sum_{i=1}^2 a_{2i}(\xi_{\rm e},\, \hat{\xi}_{\rm h},\, \hat{\eta}) u_i, \\ 
\dot{\hat{\eta}} &= \hat{q}(\xi_{\rm e},\, \hat{\xi}_{\rm h},\, \hat{\eta}) \label{eq:internal_dynamics_redefined}
\end{align}
\end{subequations}
where $\xi_{{\rm e}i}$ and $\hat{\xi}_{{\rm h}i}$ are given by $\xi_{\rm e}=[\xi_{\rm e1},\,\dots,\,\xi_{\rm e5}]^\top$ and $\hat{\xi}_{\rm h}=[\hat{\xi}_{\rm h1},\,\hat{\xi}_{\rm h2},\,\hat{\xi}_{\rm h3}]^\top$, and the functions $b_1$, $b_2$, $a_{11},\dots a_{22}$ are defined as follows: for $i, j=1,2$, 
\begin{subequations}
\begin{align}
 b_1 &:= (L_f^5 h_{\rm e}\circ \hat{\Phi}^{-1}), \\
 b_2 &:= (L_f^3 \hat{h}_{\rm h}\circ \hat{\Phi}^{-1}), \\ 
 \{ a_{ij}\}  &= \hat{{\sf A}} \notag
 \\ & := 
	\begin{bmatrix} L_{g_1}L_f^4 h_{\rm e} \circ \hat{\Phi}^{-1} &L_{g_2}L_f^4 h_{\rm e} \circ \hat{\Phi}^{-1}\\
 L_{g_1} L_f^2 \hat{h}_{\rm h} \circ \hat{\Phi}^{-1} & L_{g_2} L_f^2 \hat{h}_{\rm h} \circ \hat{\Phi}^{-1} \end{bmatrix}.
\end{align}
\end{subequations}
and $\hat{q}$ by
\begin{align}
 \hat{q}(\xi_{\rm e},\, \hat{\xi}_{\rm h},\, \hat{\eta}) := (\hat{r}\circ \hat{\Phi}^{-1})(\xi_{\rm e},\, \hat{\xi}_{\rm h},\, \hat{\eta})
\end{align}
with
\begin{align} \label{eq:function_q_redefined}
 \hat{r}(\cdot) =\begin{bmatrix}
 f_{\rm g3}(\cdot),\, f_{\rm e1}(\cdot),\, f_{\rm e2}(\cdot),\,  f_{\rm h1}(\cdot), \, f_{\rm h2}(\cdot) \end{bmatrix}^\top. 
\end{align}

As well as in \secref{sec:analysis_non_minimum_phase}, the following lemma holds for the internal dynamics:
\begin{lem} \label{th:internal_dynamics}
 The internal dynamics \eqref{eq:internal_dynamics_redefined} are independent of the variable $\hat{\xi}_{\rm h1}$. 
\end{lem}

\begin{proof}
See Appendix~\ref{sec:proof_internal_dynamics}. \qed 
\end{proof}

Thus, the minimum phase property under the outputs in \eqref{eq:redefined_output} does not depend on the reference $\hat{Y}_2^{\rm ref}$ of the output $\hat{y}_2$ and is characterized by the following theorem: 

\begin{theorem} \label{thm:minimum_phase}
 Consider the system \eqref{eq:system_equation} with the outputs in \eqref{eq:redefined_output}. 
Suppose that for given references $Y_1^{\rm ref}$ and $\hat{Y}_2^{\rm ref}$, there exists an equilibrium point $x^\ast$ of \eqref{eq:system_equation} satisfying $h_{\rm e}(x^\ast)=Y_1^{\rm ref}$ and $\hat{h}_{\rm h}(x^\ast)= \hat{Y}_2^{\rm ref}$. 
Then, the system is minimum phase at $x^\ast$ if (i) $x^\ast$ exists in the open set $\hat{D}$ stated by Lemma\,\ref{lem:redefined_normal_form}, and (ii) all the eigenvalues of the matrix ${\sf Q}$ given by 
\begin{align} \label{eq4:matrix_Q}
  {\sf Q} :=\dfrac{\partial \hat{q}}{\partial \hat{\eta}}([Y_1^{\rm ref},\, 0,\,\dots,\,0 ]^\top,\, [0,\, 0,\, 0]^\top,\, \eta^{\ast}), 
\end{align}
have negative real parts.
\end{theorem}

\begin{proof}
The proof follows directly from the definition of minimum phase system, Lemmas~\ref{lem:redefined_normal_form} and \ref{th:internal_dynamics}. 
The detail is omitted.  \qed
\end{proof}

\subsection{Synthesis of tracking controller}

Here, we synthesize the controller in \figref{fig:control_scheme}. 
First, for the reference generation, we consider the following system described by $\sigma_1$ and $\sigma_2$: 
\begin{subequations} \label{eq:reference_generation}
\begin{align}
 \dot{\sigma}_1 &= K\sigma_2,  \\
 \dot{\sigma}_2 &= h_{\rm h}(\hat{\eta})-Y_{\rm 2}^{\rm ref} 
\end{align}
\end{subequations}
where $K$ is a feedback gain, and the function $h_{\rm h}$ stands for the heat transfer rate as given by \eqref{eq:output}.
Here, we abuse the notation of $h_{\rm h}$ and regard it as a function of the variable $\hat{\eta}$ describing the internal dynamics of the system with respect to the (redefined) outputs given by \eqref{eq:redefined_output}.  
This is consistent with the output redefinition method in \cite{fiorentini12} in that the regulated output is converted into a state trajectory of new zero dynamics. 
By using the system \eqref{eq:reference_generation}, the reference $\hat{y}_2^{\rm ref}$ is given as 
\begin{align} \label{eq:reference}
 \hat{y}_{\rm 2}^{\rm ref}(t) &= \sigma_1(t). 
\end{align}
To track the references $y_1^{\rm ref}$ and $\hat{y}_2^{\rm ref}$ for the virtual outputs, various design tools are applicable.
Here, for simplicity, we utilize a standard tracking controller \cite{sastry99} by pole placement through input-output linearization: 
\begin{align} 
\label{eq:control_law} 
 \begin{bmatrix} u_1 \\ u_2 \end{bmatrix}= 
 \hat{{\sf A}}^{-1}(x) 
 \begin{bmatrix}
  -L_f^5 h_{\rm e}(x) \displaystyle{-\sum_{j=1}^5} \alpha_{{\rm e}j} \tilde{\xi}_{{\rm e}i}(t) \\  
  -L_f^3\hat{h}_{\rm h}(x) \displaystyle{-\sum_{j=1}^3} \alpha_{{\rm h}j} \tilde{{\xi}}_{{\rm h}i}(t) 
 \end{bmatrix}
\end{align}
with $\tilde{\xi}_{\rm e}$ and $\tilde{\xi}_{\rm h}$ given by 
\begin{subequations} \label{eq:error_variables}
\begin{align} 
 \tilde{\xi}_{\rm e} &:= 
 \begin{bmatrix}
  \xi_{\rm e1}-Y_{\rm 1}^{\rm ref},\, \xi_{\rm e2}, \,\dots,\, \xi_{\rm e5} 
 \end{bmatrix}^\top \\
 \tilde{{\xi}}_{\rm h} &:= 
 \begin{bmatrix}
  \hat{\xi}_{\rm h1} -\sigma_1 \\
  \hat{\xi}_{\rm h2} -K\sigma_2 \\
  \hat{\xi}_{\rm h3} -K(h_{\rm h}(\hat{\eta})-Y_{\rm 2}^{\rm ref}) 
 \end{bmatrix}, 
\end{align}
\end{subequations}
where the coefficients $\alpha_{{\rm e}i}$ and $\alpha_{{\rm h}i}$ are chosen so that the polynomials  $s^5+\alpha_{\rm e5}s^4+\dots+\alpha_{\rm e1}$ and $s^3+\alpha_{\rm h3}s^2+\dots+\alpha_{\rm h1}$ are Hurwitz.  
Then, we obtain the following theorem as the main result, stating that the precision tracking is achieved for constant references of the outputs: 
\begin{theorem} \label{thm:controller}
 Consider the model \eqref{eq:system_equation} with the control law \eqref{eq:control_law}. 
 Suppose that for given (constant) references $Y_1^{\rm ref}$ and $Y_2^{\rm ref}$, there exists a solution $(\hat{\eta}^{\rm ref},\, \sigma_2^{\rm ref})$ of the following equations:
\begin{subequations} \label{eq:assumptions_controller}
\begin{align} 
 &q([Y_{\rm 1}^{\rm ref},\, 0, \dots,\,  0]^\top,\, [0,\, K\sigma_2^{\rm ref},\, 0]^\top,\, \hat{\eta}^{\rm ref})=0, \\
 &h_{\rm h}(\hat{\eta}^{\rm ref}) - Y_2^{\rm ref} =0. 
\end{align}
\end{subequations}
Furthermore, suppose that for the open set $\hat{D}$ stated in Lemma\,\ref{lem:redefined_normal_form}, the following condition holds:
\begin{align}  \label{eq:d_hat}
 ([Y_1^{\rm ref},\,0,\dots,\,0]^\top,\,[0,\, K\sigma_2^{\rm ref},\,0]^\top,\, \hat{\eta}^{\rm ref} ) \in \hat{D}.
\end{align}
Then, there exists a trajectory of the closed-loop system satisfying 
\begin{align} \label{eq:closed-loop_trajectory}
 & (\xi_{\rm e}(t),\, \hat{\xi}_{\rm h}(t),\, \hat{\eta}(t)) \notag \\ & 
 =([Y_1^{\rm ref},\,0,\,\dots,\,0]^\top,\, [\sigma_1(0)+(K\sigma_2^{\rm ref})t,\,K\sigma_2^{\rm ref},\,0]^\top,\, \hat{\eta}^{\rm ref}).
\end{align} 
In addition to the existence of trajectory, for the trajectory starting from a point sufficiently close to \eqref{eq:closed-loop_trajectory}, the following equation holds: 
\begin{align}
 \lim_{t\to\infty} (y_1(t),\, y_2(t)) = (Y_1^{\rm ref},\, Y_2^{\rm ref}), 
\end{align}
if all of the eigenvalues of the matrix given by 
 \begin{align} \label{eq:matrix_determining_stability}
 \begin{bmatrix}
  \tilde{\sf Q}+ K B_3 C& K B_2\\
  C & 0
 \end{bmatrix} 
 \end{align}
have negative real parts, where $\tilde{{\sf Q}}$, $B_2$, $B_3$, $C$ are given by
\begin{subequations}
\begin{align} \label{eq:tilde_q}
 \tilde{\sf Q} &:= \dfrac{\partial \hat{q}}{\partial \hat{\eta} }([Y_{\rm 1}^{\rm ref},\, 0, \dots,\,  0]^\top,\, [0,\, K\sigma_2^{\rm ref},\, 0]^\top,\, \hat{\eta}^{\rm ref}), \\
 B_2 &:= \dfrac{\partial \hat{q}}{\partial \hat{\xi}_{\rm h2}}([Y_{\rm 1}^{\rm ref},\, 0, \dots,\,  0]^\top,\, [0,\, K\sigma_2^{\rm ref},\, 0]^\top,\, \hat{\eta}^{\rm ref}), \\  
 B_3 &:=\dfrac{\partial \hat{q}}{\partial \hat{\xi}_{\rm h3}}([Y_{\rm 1}^{\rm ref},\, 0, \dots,\,  0]^\top,\, [0,\,  K\sigma_2^{\rm ref},\, 0]^\top,\, \hat{\eta}^{\rm ref}), \\
 C  &:= \dfrac{\partial h_{\rm h}}{\partial \hat{\eta}}([Y_{\rm 1}^{\rm ref},\, 0, \dots,\,  0]^\top,\, [0,\, K\sigma_2^{\rm ref},\, 0]^\top,\, \hat{\eta}^{\rm ref}). 
\end{align}
\end{subequations}
\end{theorem}
\begin{proof}
With the variables $\tilde{\xi}_{\rm e}$ and $\tilde{{\xi}}_{\rm h}$ given by \eqref{eq:error_variables}, the closed-loop system is written as follows:
\begin{subequations} \label{eq:closed-loop}
\begin{align}
 \dot{\tilde{\xi}}_{\rm e} &= {\sf A}_{\rm e}  \tilde{\xi}_{\rm e}, \\
 \dot{\tilde{{\xi}}}_{\rm h} &= {\sf A}_{\rm h} \tilde{{\xi}}_{\rm h}, \\
 \dot{\hat{\eta}} &= \hat{q}(\tilde{\xi_{\rm e}}+[Y_{\rm 1}^{\rm ref},\, 0, \dots,\, 0]^\top, \nonumber \\& \makebox[8mm]{}
 \tilde{\xi}_{\rm h}+[\sigma_1,\, K\sigma_2,\, K(h_{\rm h}(\hat{\eta})-Y_{\rm 2}^{\rm ref})]^\top), \\
 \dot{\sigma}_1 &= K\sigma_2, \\
 \dot{\sigma}_2 &= K(h_{\rm h}(\hat{\eta})-Y_{\rm 2}^{\rm ref}),
\end{align}
\end{subequations}
where ${\sf A}_{\rm e}$ and ${\sf A}_{\rm h}$ are given by 
\begin{subequations}
\begin{align}
 {\sf A}_{\rm e} &:=
 \begin{bmatrix}
	0 & 1 & 0 & 0 & 0 \\
	0 & 0 & 1 & 0 & 0 \\
	0 & 0 & 0 & 1 & 0 \\
	0 & 0 & 0 & 0 & 1 \\
	-\alpha_{\rm e1} & -\alpha_{\rm e2} & -\alpha_{\rm e3} &  -\alpha_{\rm e4} & -\alpha_{\rm e5}
 \end{bmatrix}, \\
 {\sf A}_{\rm h} &:=
 \begin{bmatrix}
	0 & 1 & 0  \\
	0 & 0 & 1  \\
	-\alpha_{\rm h1} & -\alpha_{\rm h2} & -\alpha_{\rm h3} \\
 \end{bmatrix}.
\end{align}
\end{subequations}
From \eqref{eq:closed-loop}, it is verified that there exists a trajectory satisfying \eqref{eq:closed-loop_trajectory}. 
Furthermore, from Lemma\,\ref{th:internal_dynamics}, the dynamics of $\sigma_1$ can be separated from the rest of the system because $\sigma_1$ affects the value of $\xi_{\rm h1}$ only (see \figref{fig:error_system}). 
Thus, we have 
\begin{subequations} \label{eq:error_system} 
\begin{align} 
 \dot{\tilde{\xi}}_{\rm e} &= {\sf A}_{\rm e}  \tilde{\xi}_{\rm e}, \\
 \dot{\tilde{{\xi}}}_{\rm h} &= {\sf A}_{\rm h} \tilde{{\xi}}_{\rm h}, \\
 \dot{\hat{\eta}} &= \hat{q}(\tilde{\xi_{\rm e}}+[Y_{\rm 1}^{\rm ref},\, 0, \dots,\, 0]^\top, \nonumber \\& \makebox[8mm]{}
 \tilde{\xi}_{\rm h}+[0,\, K\sigma_2,\, K(h_{\rm h}(\hat{\eta})-Y_{\rm 2}^{\rm ref})]^\top), \\
 \dot{\sigma}_2 &= K(h_{\rm h}(\hat{\eta})-Y_{\rm 2}^{\rm ref}). 
\end{align}
\end{subequations}
For the error system \eqref{eq:error_system},  the trajectory given by \eqref{eq:closed-loop_trajectory} is an equilibrium point. 
The stability of the equilibrium point is determined  by the following matrix: 
 \begin{align} 
 \begin{bmatrix}
  {\sf A}_{\rm e} & 0 & 0 & 0 \\
  0 & {\sf A}_{\rm h} & 0 & 0 \\
  \ast  & \ast & \tilde{{\sf Q}}+ K B_3 C& K B_2\\
  \ast & \ast & C & 0
 \end{bmatrix}. 
 \end{align}
By the assumption, all of the eigenvalues of the above matrix have negative real parts, and the proof is thus 
completed. \qed
\end{proof}

\begin{figure}[t!]
\centering 
\includegraphics[width=0.6\hsize]{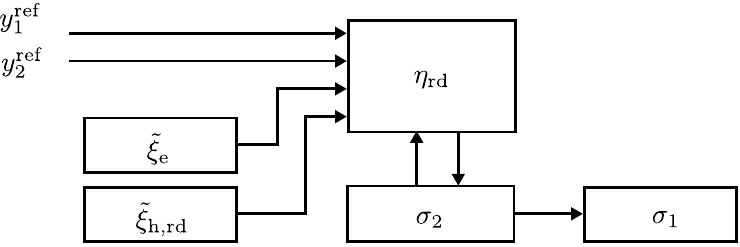}
\caption{Graphical representation of the closed-loop system}
\label{fig:error_system}
\end{figure}

\begin{rem}
For given references $Y_1^{\rm ref}$ and $Y_2^{\rm ref}$, the assumptions \eqref{eq:assumptions_controller}, \eqref{eq:d_hat} and \eqref{eq:matrix_determining_stability} in Theorem~\ref{thm:controller} can be evaluated in the original coordinate $x$ to achieve the control objective in Problem~\ref{prob:tracking}. 
In addition to the primary control objective to regulate the energy flows, the proposed controller can be utilized for stabilizing an equilibrium point by setting $K=0$ in \eqref{eq:reference_generation}. 
With $K=0$, the references of $y_1$ and $\hat{y}_2$ are given as $y_1^{\rm ref}(t)=Y_1^{\rm ref}$ and $\hat{y}_2^{\rm ref}(t)=\sigma_1(0)$, and the assumptions in Theorem~\ref{thm:controller} correspond to those of Theorem~\ref{thm:minimum_phase}.  
Thus, the control law \eqref{eq:control_law} is a standard stabilizing controller for a minimum phase system. 
\end{rem}

\section{Numerical Simulation} \label{sec:simulation}

This section demonstrates the effectiveness of the control law \eqref{eq:control_law} via numerical simulation.  
The values of the parameters are based on \cite{rowen83,machowski08,astrom00,bujak09}, and are shown in \tabref{tab:parameters}. 

\begin{figure}[t!]
\centering
{ %
\includegraphics[width=0.55\hsize]{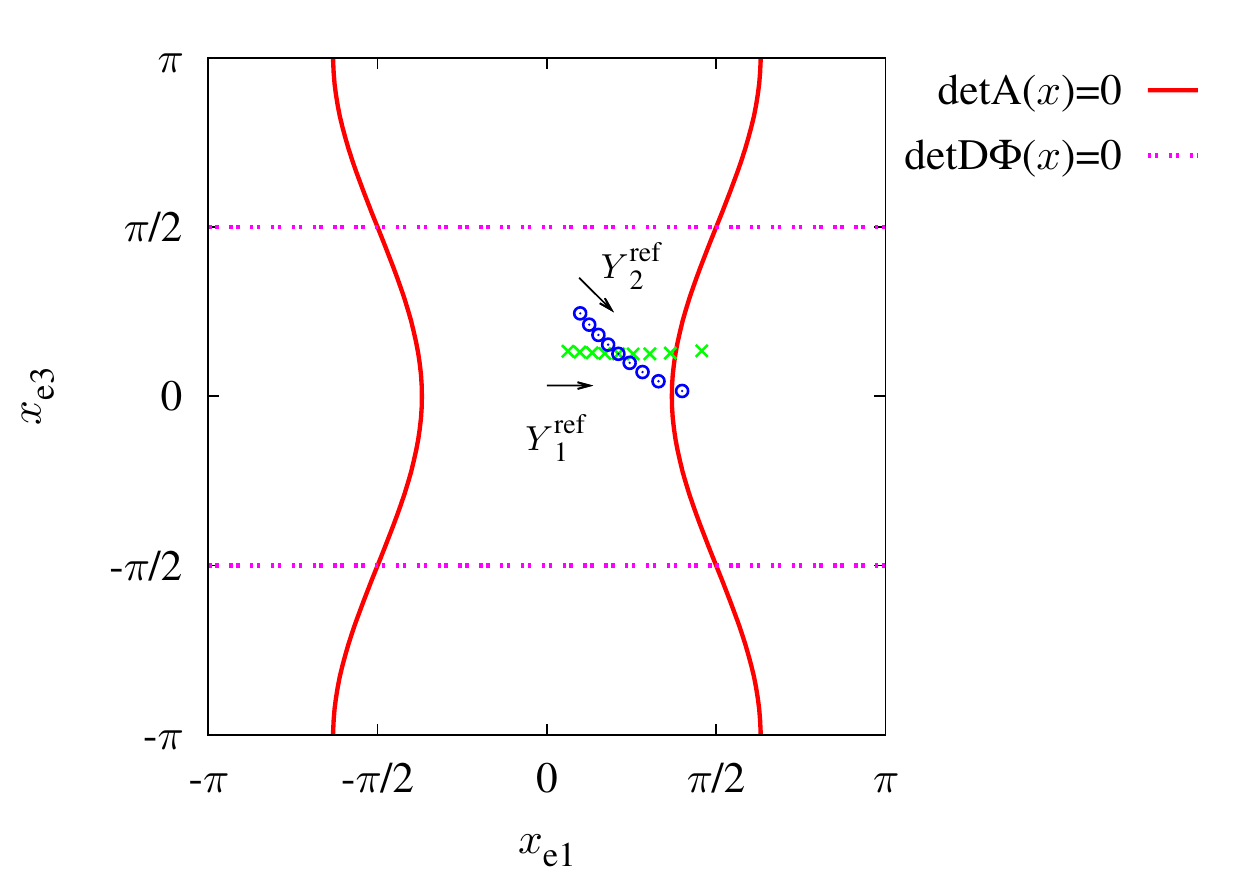} \\
\subcaption{Singularity conditions of ${\sf A}(x)$ and ${\rm D}\Phi(x)$} \label{fig:singularity}
\hspace*{2mm}\includegraphics[width=0.6\hsize]{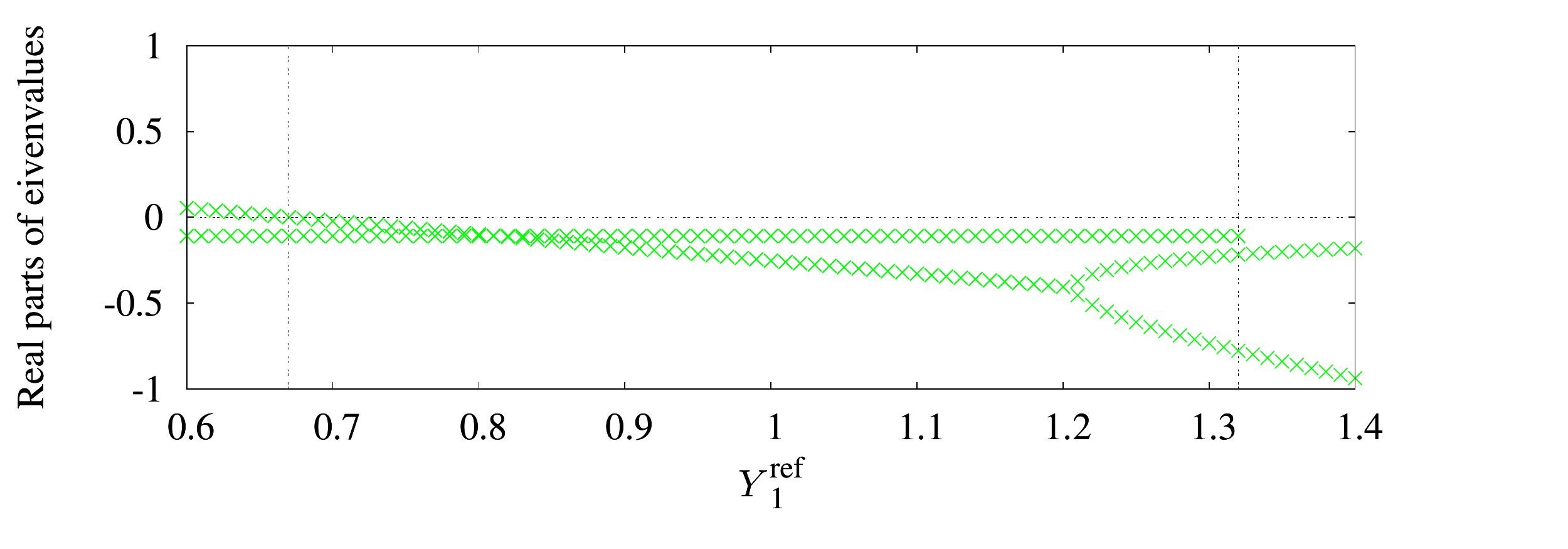}\\[-2mm]
\subcaption{%
Real parts of eigen values of ${\sf Q}$ 
} \label{fig:eigenvalues} 
\hspace*{2mm}\includegraphics[width=0.6\hsize]{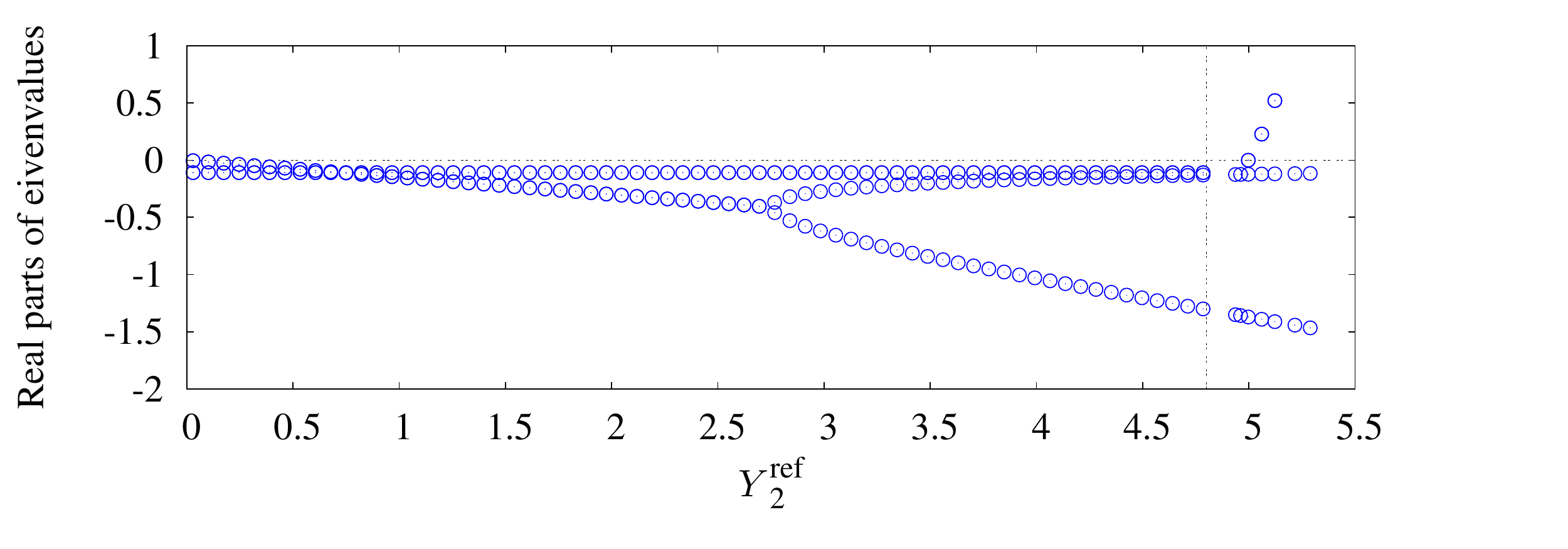}\\[-2mm]
\subcaption{Real parts of eigen values of $\tilde{\sf Q}$ 
} \label{fig:eigen_internal}
} %
\caption{Numerical simulations for checking assumptions in Theorems~\ref{thm:minimum_phase} and \ref{thm:controller} for various $Y_1^{\rm ref}$ and ${Y}_2^{\rm ref}$}
\label{fig:minimum_phase}
\end{figure}

To demonstrate the effectiveness of the controller in Theorem~\ref{thm:controller}, we first confirm the assumptions of Theorem~\ref{thm:minimum_phase} characterizing the minimum phase property of the model \eqref{eq:system_equation} with the outputs in \eqref{eq:redefined_output}. 
For this, \figref{fig:minimum_phase} shows (a) singularity conditions for $\hat{\sf A}(x)$ and ${\rm D}\hat{\Phi}(x)$ and (b)  eigenvalues of the matrix ${\sf Q}$. 
Based on \figref{fig:singularity}, the open set $\hat{D}$ stated in Lemma~\ref{lem:redefined_normal_form} can be considered.  
The \emph{solid} lines show the condition ${\rm det}\hat{{\sf A}}(x)=0$ and the \emph{broken} lines ${\rm det}{\rm D}\hat{\Phi}(x)=0$. 
Since the conditions depend only on $x_{\rm e1}$ and $x_{\rm e3}$, by defining a set $\hat{D}_{\rm e} \subset \mathbb{R}^2$ as the inside of these lines containing $(x_{\rm e1},\, x_{\rm e3})=(0,\,0)$, the set $\hat{D}$ is given as follows:
\begin{align}
 \hat{D} = \left\{ \left[x_{\rm g}^\top ,\,x_{\rm e}^\top ,\,x_{\rm h}^\top \right]^\top \in X~|~ (x_{\rm e1}, x_{\rm e3}) \in \hat{D}_{\rm e} \right\}.
\end{align}
In \figref{fig:singularity}, the points ($\times$) show the values of $x_{\rm e1}$ and $x_{\rm e3}$ on an equilibrium point under various settings of $Y_{\rm 1}^{\rm ref}$.
Note that they are independent of $Y_2^{\rm ref}$ due to Lemma~\ref{th:internal_dynamics}. 
The equilibrium points exist in the open set $\hat{D}$ when $Y_1^{\rm ref} \in [0.6,\,1.32]$. 
For the equilibrium points, it is confirmed that the matrix ${\sf Q}$ has two pairs of complex conjugate eigenvalues and one real eigenvalue for $Y_1^{\rm ref}<1.2$.
Figure~\ref{fig:eigenvalues} shows the real parts of the two sets of eigenvalues, and the real eigenvalue exists around $-10$. 
For $Y_1^{\rm ref} < 0.67$, a set of eigenvalues has positive real parts.
Thus, the model \eqref{eq:system_equation} is a minimum phase system under $Y_1^{\rm ref} \in [0.67,\,1.32]$. 

Next, we consider the assumptions of Theorem~\ref{thm:controller}. 
In \figref{fig:singularity}, the circles ($\circ$) show the values of $(x_{\rm e1},x_{\rm e3})$ for various $Y_2^{\rm ref}$ with $Y_1^{\rm ref}=1.0$. 
It is confirmed that the point given in \eqref{eq:d_hat} exists in the open set $\hat{D}$ for $Y_2^{\rm ref} \in [0,\,4.85]$. 
Furthermore, with respect to the convergence of the trajectory \eqref{eq:closed-loop_trajectory}, the eigenvalues of the matrix \eqref{eq:matrix_determining_stability} can be directly checked as conducted above for Theorem~\ref{thm:minimum_phase}.  
However, since the proposed controller is given by a simple augmentation with the integral error, here we discuss the eigenvalues of the matrix \eqref{eq:matrix_determining_stability} through the matrix $\tilde{\sf Q}$ and the feedback gain $K$. 
Particularly, for $K=0$, the eigenvalues of the matrix \eqref{eq:matrix_determining_stability} are those of $\tilde{\sf Q}$ ($={\sf Q}$) and 0. 
Thus, for $K\neq 0$, if all the eigenvalues of $\tilde{\sf Q}$ have negative real parts, it is expected to choose a small value of $K$ so that the eigenvalues of the matrix \eqref{eq:matrix_determining_stability} have negative real parts. 
Figure~\ref{fig:eigen_internal} shows the real parts of the eigenvalues of $\tilde{\sf Q}$ under $Y_1^{\rm ref}=1.0$ and various $Y_2^{\rm ref}$. 
It is confirmed that for  $Y_2^{\rm ref} < 4.85$, the all the eigenvalues has negative real parts.

\begin{figure}[t!]
\centering
{\includegraphics[width=0.6\hsize]{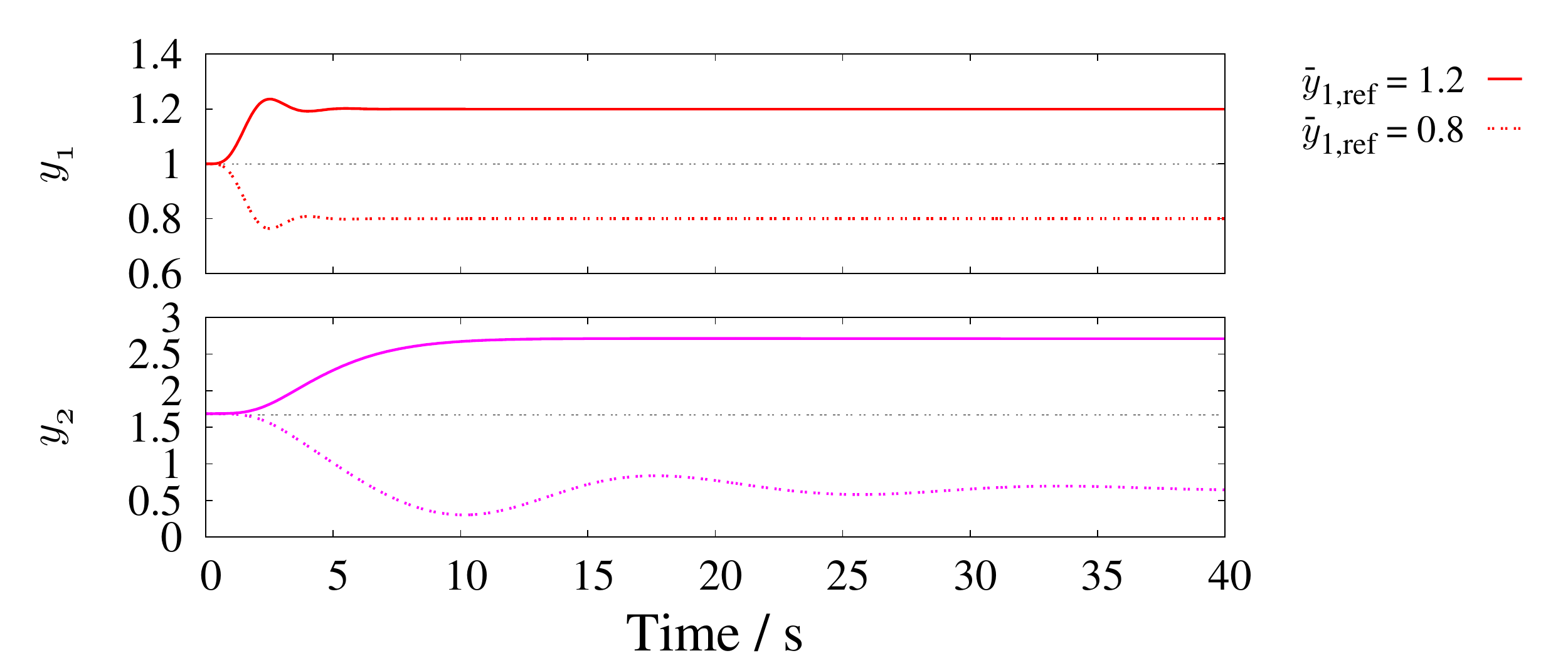}\\[-1mm]
\subcaption{Output variables}
\includegraphics[width=0.6\hsize]{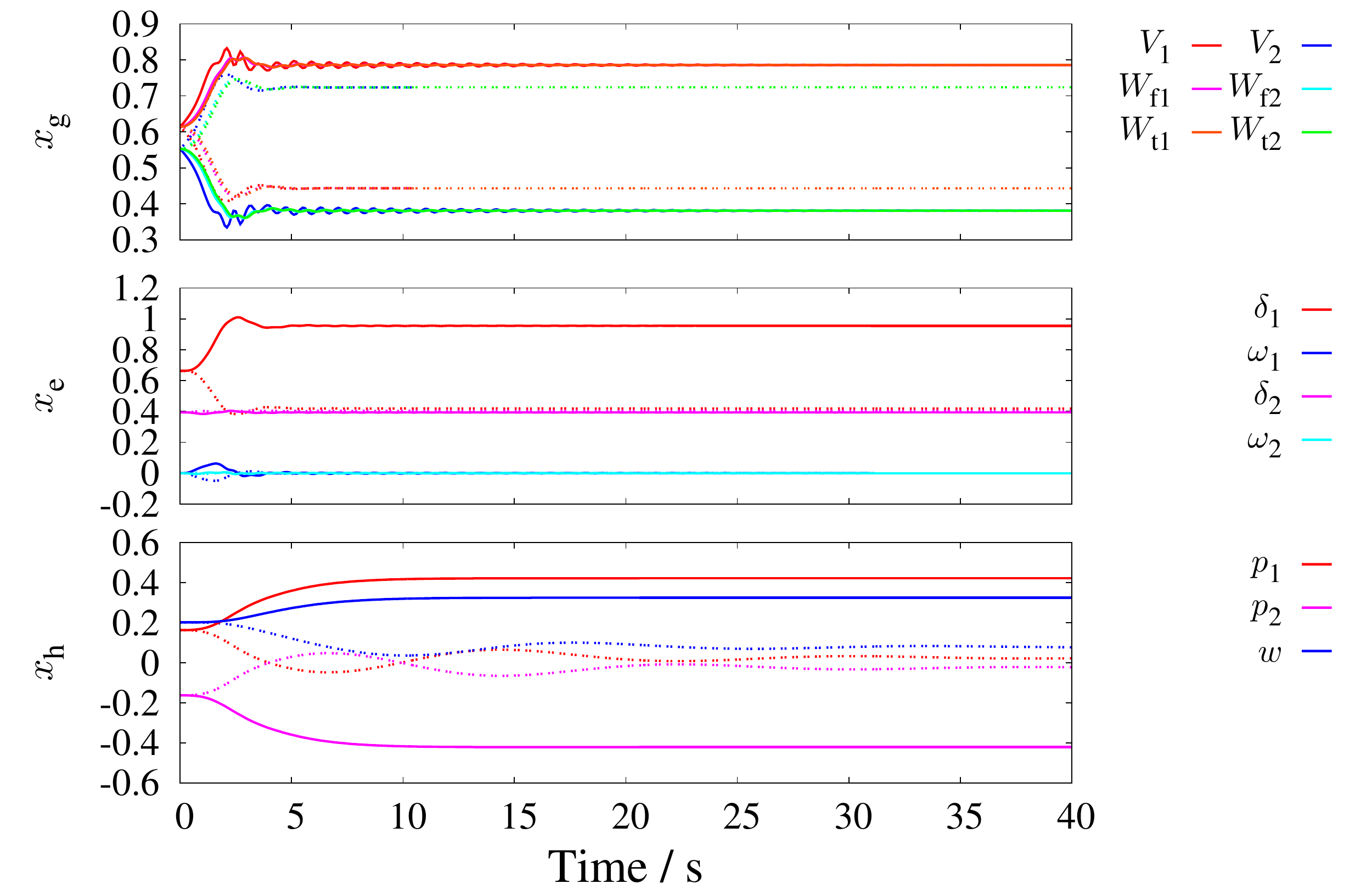}\\[-1mm]
\subcaption{State variables}
\includegraphics[width=0.6\hsize]{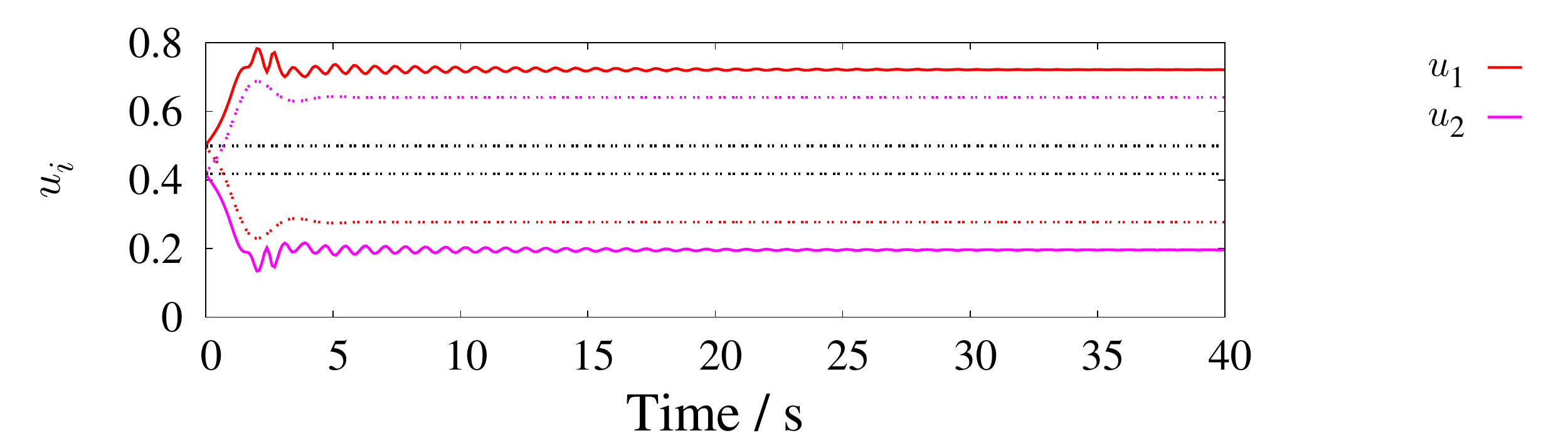}\\[-1mm]
\subcaption{Input variables}}
\caption{Time responses of (a) output, (b) state, and (c) input variables of the model. The solid lines show the responses with  $\hat{y}_{\rm 1, ref}=1.2$, and the broken lines $\hat{y}_{\rm 1, ref}=0.8$.}
\label{fig:simulation}
\end{figure}

Finally, we demonstrate the effectiveness of the control law \eqref{eq:control_law}. 
For $K=0$, the control law stabilizes the equilibrium point determined by $Y_1^{\rm ref}$ and $\hat{Y}_2^{\rm ref}=\sigma_1(0)$. 
Figure~\ref{fig:simulation} shows the time responses of the closed-loop system with $(Y_1^{\rm ref}, \hat{Y}_2^{\rm ref})=(1.2, 0)$ and $(0.8, 0)$. 
The solid lines represent the time responses with $Y_1^{\rm ref}=1.2$, and the broken lines $Y_1^{\rm ref}=0.8$. 
The initial conditions for the numerical simulation are set to the values at the equilibrium point with $(Y_1^{\rm ref}, \hat{Y}_2^{\rm ref})=(1.0, 0)$. 
The coefficients $\alpha_{\rm e1},\,\dots,\, \alpha_{\rm h3}$ of the control law \eqref{eq:control_law} are chosen as follows:
\begin{subequations}
\begin{align}
 & s^5+\alpha_{\rm e5}s^4+\cdots+\alpha_{\rm e1}s+\alpha_{\rm e1} 
 \notag \\ & \hspace{20mm} 
= (s^2+5 s+2.5^2)^2(s+2.5), \\
 & s^3+\alpha_{\rm h3}s^2+ \alpha_{\rm h2}s +\alpha_{\rm h1}s+\alpha_{\rm h1} 
 \notag \\ & \hspace{20mm} 
= (s^2+0.5s+0.25^2)(s+0.25).
\end{align}
\end{subequations}
Thus, the time constants of the closed-loop system are set to $1/2.5 \,{\rm s}$ for $y_1$ and $1/0.25\, {\rm s}$ for $y_2$. 
These constants were chosen by trial and error so that the inputs $u_1,\, u_2$ and the variable $x_{\rm g}$ were kept within their nominal range $[0,\,1]$. 
The simulation result demonstrates that the controller asymptotically stabilizes the corresponding equilibrium point. 
Next, for $K\neq 0$, the control law \eqref{eq:control_law} achieves the objective of regulation in Problem~\ref{prob:tracking}. 
Figure~\ref{fig:tracking} shows the responses of the closed-loop system under $(Y_1^{\rm ref},\,Y_2^{\rm ref})=(1.2, 1.69)$. 
The initial conditions and the coefficients  $\alpha_{\rm e1},\,\dots,\, \alpha_{\rm h3}$ are the same as in \figref{fig:simulation}. 
It is confirmed that the regulation of energy flows $y_1$ and $y_2$ is achieved by choosing a small value of $K$. 
{\color{black}
This is consistent with the fact that the reference $\hat{y}_2^{\rm ref}$ slowly changes on the desired invariant manifold to be stabilized.}

\begin{figure}[t!]
\centering
\includegraphics[width=0.6\hsize]{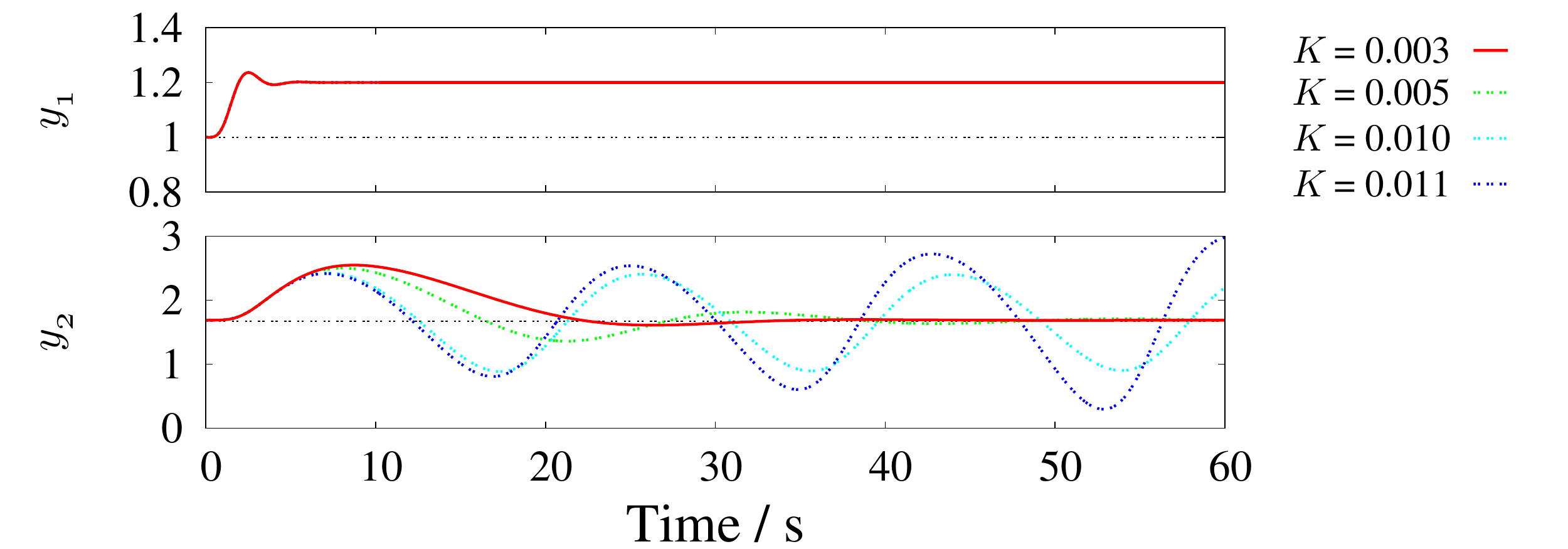}
\caption{Time responses of the outputs under the control law \eqref{eq:control_law} with $K\neq 0$}
\label{fig:tracking}
\end{figure}

\section{Conclusions} \label{sec:conclusion}

This paper introduced a control problem to regulate energy flows in a two-site electricity and heat supply system as a typical and central situation in the next-generation interconnected energy systems.  
We performed a structural analysis of the state-space model of the two-site system through  input-output linearization. 
When the outputs of the state-space model are given as the energy flows in order to regulate these values to given references, the system becomes a non-minimum phase system due to nonexistence of isolated asymptotically stable equilibrium point of the zero dynamics. 
Instead, we identified an invariant manifold characterizing the stability property of the zero dynamics. 
The existence of the manifold is due to physical property of heat supply and characterizes the multiscale property of the dynamics of the controlled system.  
Then, it was shown that by choosing a virtual output parameterizing the dynamics along the invariant manifold, the model became a minimum phase system. 
Based on this, we proposed a control scheme that achieves both regulation of the energy flows and stabilization of an equilibrium point without changing its structure.  
The effectiveness of the controller was established with numerical simulation under a practical setting of model parameters. 

Finally, several future directions are summarized in the following. 
One is to provide a generalized framework of the analysis and control presented in this paper. 
Since the control problem in this paper is specialized to the minimal two-site system, it is significant to obtain a similar result for the general $n$-site system. 
For this, Lemma~\ref{thm:NHIM} can be extended to the general case as shown in \cite{cdc15,cndpaper16}. 
Another future direction is to introduce more practical considerations into the controller design for guaranteeing the robustness in the presence of anticipated uncertainty.
The approaches introduced in \cite{krstic95,sepulchre97,chen15} for nonlinear and robust control design will be considered and deployed. 
Furthermore, since the control objective in this paper is to have the electric power to the infinite bus and the heat flow rate being regulated, the references of the outputs $y_1$ and $y_2$ were constants. 
The problem formulation could be extended to have the outputs tracking time-varying reference trajectories 
for more complex system performance objectives such as compensation of variable outputs of renewable energy sources.

\begin{acknowledgment}
%
This work was partially supported by CREST (\# JPMJCR15K3), JST.  
\end{acknowledgment}

%
% ...

\bibliographystyle{asmems4}

\bibliography{dissert_hoshino_rev20180320,self} 

\appendix       
%

\begin{table}[ht] %
 \centering
 \caption{Parameters for numerical simulation} \label{tab:parameters}
 \fontsize{9}{10}\selectfont
 \begin{tabular}{c l c} \hline
Symbol & Meaning & Value \\\hline
  $T_{{\rm v}i}$ & Valve positioner time constant & $0.05\,{\rm s}$ \\
  $T_{{\rm f}i}$ & Fuel system time constant & $0.4\,{\rm s}$ \\
  $T_{{\rm CD}i}$ & Compressor volume time constant & $0.1\,{\rm s}$ \\
  $W_{{\rm o}i}$ &  Fuel valve lower limit & $0.23\,{\rm p.u.} $ \\   
  $K_{\rm e1}$ & Rated mechanical power \#1 & $7.5\,{\rm MW}$\\
  $K_{\rm e2}$ & Rated mechanical power \#2 & $3.0\,{\rm MW}$\\
  $K_{{\rm h}i}$  &Rated heat output & $6.0\,{\rm MJ/s} $\\
  $\beta_i$ & Coefficient for no fuel condition & $0$\\ 
 \hline
 $\omega_{\rm s}$ & Synchronous speed & 377\,{\rm rad/s} \\
 $H_i$ & Per-unit inertia time constant & $10\,{\rm s}$ \\
 $D_i$ & Damping coefficient & $0.05\,{\rm p.u.}$  \\
 $E_i$ & Voltage of the generator & $1.0\,{\rm p.u.}$ \\
 $B_{i0}$ & Transfer susceptance to infinete bus & $1.0\,{\rm p.u.} $\\
 $B_{12}$ & Transfer  susceptance & $0.5\,{\rm p.u.} $\\
 $G_{ij}$ & Transfer conductances & 0 \\
 \hline
  $p_0$ & Nominal value of pressure &$800\,{\rm kPa}$ \\
  $\rho_{\rm s}$  &Density of saturated steam &$4.161\,{\rm kg/m^3}$\\
  $h_{\rm s}$ &Specific enthalpy of steam &2768\,{\rm kJ/kg}\\
  $h_{\rm w}$ &Specific enthalpy of water & 721.0\,{\rm kJ/kg}\\
  $e_i$ &Coefficient of pressure variation & $3073\,{\rm J/Pa}$\\
  $d$&Diameter of the steam pipe&$0.2\,{\rm m}$\\
  $L$ & Length of the pipe & 200\,{\rm m} \\
  $\lambda$&Friction coefficient&0.016\\
  $Q'_{{\rm L}1}$&Heat consumption at site \#1&$2.0\,{\rm MJ/s}$ \\
  $Q'_{{\rm L}2}$&Heat consumption at site \#2&$5.0\,{\rm MJ/s}$
 \\\hline
 \end{tabular}
\end{table}

\section{Model equations} \label{sec:model_equations}

\subsection{Gas turbine}

The dynamics of gas turbine are described by a simplified form of the model in \cite{rowen83}. 
The state variable $x_{\rm g}$ consists of,  
the variable $v_{{\rm p}i}$ standing for the position of fuel valve of the gas turbine $\# i$, and $w_{{\rm f}i}$ and $w_{{\rm t}i}$ for the fuel flow rates at the combustor and at the turbine. 
The variables are scaled by their rated values, so that they are normally in the range of $[0,1]$. 
The dynamics of the gas turbines are represented by the following equation: for $i=1,\,2$, 
\begin{subequations}
\label{eq:gas_turbine_model}
\begin{align}
 \dot{v}_{{\rm p}i} &= \dfrac{1}{T_{{\rm v}i}}(-v_{{\rm p}i} + W_{{\rm o}i}) + \dfrac{1-W_{{\rm o}i}}{T_{{\rm v}i}} u_i ,\\
 \dot{w}_{{\rm f}i}  &= \dfrac{1}{T_{{\rm f}i}} ( -w_{{\rm f}i} +v_{{\rm p}i} ), \\ %
 \dot{w}_{{\rm t}i} &=  \dfrac{1}{T_{{\rm CD}i}} ( -w_{{\rm t}i}+w_{{\rm f}i} ) 
\end{align}
\end{subequations}
where $T_{{\rm v}i}$, $T_{{\rm f}i}$, and $T_{{\rm CD}i}$ stand for the time constants of the valve positioning system, the down stream piping and fuel gas distribution, and the compressor discharge. 
The parameter $W_{{\rm o}i}$ stands for the lower limit of the fuel valve. 
The mechanical power $P_{{\rm m}i}$ and the heat flow rate $Q'_{{\rm a}i}$ are given as follows:
\begin{subequations} \label{eq:electricity_to_heat_ratio}
\begin{align}
 P_{{\rm m}i} & := K_{{\rm e}i} \left( w_{{\rm t}i} - W_{{\rm o}i} \right)/(1-W_{{\rm o}i}), 
 \label{eq:gasturbine_output} \\
 Q'_{{\rm a}i} & := K_{{\rm h}i}  (w_{{\rm f}i} +\beta_i)/(1+\beta_i) 
\end{align}
\end{subequations}
where $K_{{\rm e}i}$ and $K_{{\rm h}i}$ stand for the rated outputs of mechanical power and heat flow rate.  
The parameter $K_{{\rm h}i}\beta_i/(1+\beta_i)$ represents the heat output with no fuel input. 
By defining $ x_{\rm g} := [ v_{\rm p1} ~ w_{\rm f1} ~ w_{\rm t1} ~ v_{\rm p2} ~ w_{\rm f2} ~ w_{\rm t2} ]^\top$, the functions $f_{\rm g}$ and $g_{{\rm g}i}$ in \eqref{eq:system_equation} are given as follows:
\begin{align}
 f_{\rm g}(x_{\rm g}) := 
\begin{bmatrix}
    (-x_{{\rm g}1} + W_{{\rm o}1}) / T_{{\rm v}1} \\ 
    ( -x_{{\rm g}2} +x_{{\rm g}1} )  / T_{{\rm f}1} \\
    ( -x_{{\rm g}3}+x_{{\rm g}2} )/  T_{{\rm CD}1} \\
    (-x_{{\rm g}4} + W_{{\rm o}2}) / T_{{\rm v}2} \\ 
    ( -x_{{\rm g}5} +x_{{\rm g}4} )  / T_{{\rm f}2} \\
    ( -x_{{\rm g}6}+x_{{\rm g}5} )/  T_{{\rm CD}2}
\end{bmatrix}, 
\end{align}
\begin{align}
 g_{\rm g1} := 
\begin{bmatrix} (1-W_{{\rm o}1})/T_{{\rm v}1} & 0 & 0& 0& 0 & 0 \end{bmatrix}^\top, \\
 g_{\rm g2} := 
\begin{bmatrix} 0 & 0 & 0& (1-W_{{\rm o}2})/T_{{\rm v}2} & 0 & 0 \end{bmatrix}^\top.
\end{align}

\subsection{Electric subsystem}
The model of electric subsystem is based on the swing equation \cite{machowski08}. 
The state variable $x_{\rm e}$ consists of the electric angular position $\delta_i$ of rotor with respect to the infinite bus and the deviation of rotor speed  $\omega_{i}$ relative to the synchronous speed $\omega_{\rm s}$ for each generator. 
The variable $\delta_i$ is in the electrical radian, and $\omega_i$ is scaled by $\omega_{\rm r}:=\sqrt{\omega_{\rm s}/2H_i}$, where $H_i$ stands for the per-unit time constant of rotor. 
The dynamics of the electric sub-system are represented as follows: for $i=1,\,2$, 
\begin{subequations}
  \label{eq:electricity_model}
  \begin{align}
    \dot{\delta_i} =& \dfrac{1}{T_{{\rm e}i}} \omega_i \\
   \dot{\omega_i} =& \dfrac{1}{T_{{\rm e}i}} \{P_{{\rm m}i}(x_{\rm g}) -D_{i} \omega_i -P_{{\rm e}i}(\delta_1,\,\delta_2) \}  
  \end{align}
\end{subequations}
where $T_{{\rm e}i} :=1/\omega_{\rm r}$, and $D_i$ stands for the damping coefficient. 
The function $P_{{\rm e}i}(\delta_1,\,\delta_2)$ stands for the electric output power of the generator $\# i$,  given by 
\begin{align}
& P_{{\rm e}i}(\delta_1,\delta_2)= \sum_{j\in\{1,2,\infty\}} P_{ij}(\delta_1,\delta_2)
\end{align}
with the symbol $\infty$ representing the infinite bus, and $P_{ij}(\delta_1,\delta_2)$ given by 
\begin{align}
 \label{eq:power_flows}
& P_{ij}(\delta_1,\delta_2)= E_i E_j\{G_{ij}\cos(\delta_i-\delta_j) \notag\\ &\hspace{30mm} 
+B_{ij}\sin(\delta_i-\delta_j) \} 
\end{align}
where $\delta_\infty=0$ and $E_i$ corresponds to the voltage behind synchronous reactance. 
The parameters  $G_{ij}+{\rm i} B_{ij}$ stand for the transfer admittances associated with the voltages $E_i$ and $E_j$, and the electric load consumptions and transmission losses are included in $G_{ij}$.
Similarly, the electric power to the infinite bus is described by 
\begin{align}
 P_{\rm e\infty}(\delta_1, \delta_2):= -P_{\rm \infty1}(0, \delta_1) -P_{\rm \infty2}(0, \delta_2). 
\end{align}
By defining $x_{\rm e} := [ \delta_1 ~ \omega_1 ~ \delta_2 ~ \omega_2 ]^\top$, the function $f_{\rm e}$ is given as follows: 
\begin{align} 
 &f_{\rm e}(x_{\rm e},\,x_{\rm g})  := \notag \\ &
\begin{bmatrix}  
 x_{\rm e2}/ T_{{\rm e}1} \\ 
 \{P_{{\rm m}1}(x_{\rm g3}) -D_{1} x_{\rm e2} -P_{{\rm e}1}(x_{\rm e1}, \, x_{\rm e3}) \}  / T_{{\rm e}1} \\
 x_{\rm e4}/ T_{{\rm e}2} \\ 
 \{P_{{\rm m}2}(x_{\rm g6}) -D_{1} x_{\rm e4} -P_{{\rm e}2}(x_{\rm e1}, \, x_{\rm e3}) \}  / T_{{\rm e}2}
\end{bmatrix}.
\end{align}

\subsection{Heat subsystem} \label{sec:model_heat}

For the heat subsystem, we utilize a simplified model derived in \cite{cdc15,cndpaper16} based on \cite{astrom00,kim00,osiadacz87}. 
The dynamics are described by $p_i$ standing for the deviation of boiler pressure relative to the nominal value and $w$ for the velocity of steam inside the pipe. 
The variables $w$ and $p$ are scaled by $w_{\rm r}:=Q'_{\rm r}/d^2h_{\rm r}\rho_{\rm r}$ and $p_{\rm r}:=\rho_{\rm r}w_{\rm r}^2$, where $Q'_{\rm r}$, $h_{\rm r}$, and $\rho_{\rm r}$ are rated values of the heat flow rate, enthalpy, and density. 
The parameter $d$ stands for the diameter of the pipe.
The dynamics of the heat subsystem are represented by the following equations:
\begin{subequations}
 \label{eq:heat_model}
 \begin{align}
   \dot{p}_1&= \dfrac{1}{T_{\rm h1}} \{ Q'_{\rm a1}(w_{{\rm f}i}) -Q'_{{\rm L}1} - Q'_{12}(w) \}, \\
   \dot{p}_2 &= \dfrac{1}{T_{\rm h2}} \{ Q'_{\rm a2}(w_{{\rm f}i}) -Q'_{{\rm L}2}  - Q'_{21}(w) \}, \\
   \dot{w} &= \dfrac{1}{T_{\rm h3}} \biggl( \dfrac{p_1-p_2}{\rho_{\rm s}} - \frac{\lambda L}{2d} w|w| \biggr) 
 \end{align}
\end{subequations}
where $Q'_{{\rm L}i}$ stands for the heat load at each site, $\rho_{\rm s}$ for the density of steam, $\lambda$ for the friction coefficient, and $L$ for the length of the pipe. 
The functions $Q'_{12}(w)$ and $Q'_{21}(w)$ stand for the heat flow rates through the pipe and are given as follows:
\begin{align}
  Q'_{12}(w) = -Q'_{21}(w) = \dfrac{\pi }{4}h_{\rm c}\rho_{\rm s} w 
\end{align}
where $h_{\rm c}$ stands for the enthalpy of condensation.
The time constants $T_{{\rm h}i}$ of boilers for $i=1,2$ are given by $T_{{\rm h}i}:=Q'_{\rm r} e_i /d^4 h_{\rm r}^2 \rho_{\rm r}$, where $e_i$ stands for the rate of change of internal energy stored in the boilers. 
The time constant $T_{\rm h3}$ for response of steam flow in the pipe is given as $T_{\rm h3}:=d^2L h_{\rm r}\rho_{\rm r}/Q'_{\rm r}$.  
For the analysis in \secref{sec:analysis}, where Lie derivatives with respect to the vectorfield described by \eqref{eq:heat_model} are calculated, we assume $w>0$ so as to ensure the smoothness of the vectorfield, although we do not consider the constraint on $w$ explicitly. 

For the current two-site system, it can be seen from \eqref{eq:heat_model} that the dynamics of the heat subsystem are governed by the following two variables: 
the pressure difference $p_1-p_2$ and the steam velocity $w$. 
Note that it is common for hydraulic networks to be represented by these two kinds of variables \cite{osiadacz87,persis14}. 
Based on this, here we introduce the following coordinate transformation to defined the state variable $x_{\rm h}$: 
\begin{align}
 \label{eq:transformation_decouple}
 x_{\rm h} :=
\begin{bmatrix}
 p_1-p_2 \\ w \\ \dfrac{T_{\rm h1}p_1+T_{\rm h2} p_2}{T_{\rm h1}+T_{\rm h2}} 
\end{bmatrix} 
= {\sf T}
\begin{bmatrix} p_1 \\ p_2 \\ w \end{bmatrix}
\end{align}
with 
\begin{align} 
{\sf T} := 
\begin{bmatrix}
 1 & -1  & 0 \\ 0 & 0& 1 \\
 \dfrac{T_{\rm h1}}{T_{\rm h1}+T_{\rm h2}} & \dfrac{T_{\rm h2}}{T_{\rm h1}+T_{\rm h2}} & 0 
\end{bmatrix}. 
\end{align}
In the new coordinate, the first two variables stand for the pressure difference and the velocity, and the last for the (weighted) averaged pressure level of the two boilers. 
From the above observation that the dynamics described by \eqref{eq:heat_model} is independent of $x_{\rm h3}$ representing the averaged pressure level. 
The function $f_{\rm h}$ is given as follows: 
\begin{align}
 & f_{\rm h}(x_{\rm h}, x_{\rm g})  :=  \notag \\
 & \begin{bmatrix} 
  (Q'_{\rm a1}(x_{{\rm g}2}) -Q'_{{\rm L}1} - Q'_{12}(x_{\rm h2}))/T_{\rm h1} \, ... \\ \hspace{8mm} - (Q'_{\rm a2}(x_{{\rm g}5}) -Q'_{{\rm L}2}  - Q'_{21}(x_{\rm h2}))/T_{\rm h2} \\
  \dfrac{1}{T_{\rm h3}} \biggl( \dfrac{x_{\rm h1}}{\rho_{\rm s}} - \dfrac{\lambda L}{2d} x_{\rm h2}|x_{\rm h2}|  \biggr) \\
  Q'_{\rm a1}(x_{{\rm g}2}) -Q'_{{\rm L}1} + Q'_{\rm a2}(x_{{\rm g}5}) -Q'_{{\rm L}2} 
  \end{bmatrix}. 
\end{align}

\section{Proof of Lemma~\ref{thm:normal_form}} \label{sec:proof_normal_form}

The proof follows from direct application of input-output linearization.   
The Lie derivatives needed for the linearization can be calculated with symbolic computation e.g. Mathematica. 
For the chosen output, the decoupling matrix ${\sf A}(x)$ is obtained by differentiating the outputs $y_1$ and $y_2$ in $5$ and $4$ times, respectively, i.e. for $i=1,2$, 
\begin{align}
 & L_{g_i} L_f^0 h_{\rm e}(x) \equiv \cdots \equiv L_{g_i}L_f^3 h_{\rm e}(x) \equiv   0, \\
 & L_{g_i} L_f^0 h_{\rm h}(x) \equiv L_{g_i}L_f h_{\rm h}(x) \equiv  L_{g_i}L_f^2 h_{\rm h}(x) \equiv 0,
\end{align}
and the following state-dependent matrix can be defined:  
\begin{align}
{\sf A}(x) &:= 	
\begin{bmatrix} L_{g_1}L_f^4 h_{\rm e}(x) &L_{g_2}L_f^4 h_{\rm e}(x)  \\[1.5ex]
 L_{g_1} L_f^3 h_{\rm h}(x) & L_{g_2} L_f^3 h_{\rm h}(x) 
\end{bmatrix}.  
\end{align}
The matrix ${\sf A}(x)$ depends only on $x_{\rm e}$, and its determinant is given by 
\begin{align}
\label{eq:decoupling}
 {\rm det}{\sf A}(x) = -A_{\rm 1} \dfrac{{\rm d}P_{\infty1}}{{\rm d}x_{\rm e1}}(x_{\rm e1}) -A_{\rm 2} \dfrac{{\rm d}P_{2\infty}}{{\rm d}x_{\rm e3}}(x_{\rm e3})
\end{align}
where $A_{\rm 1}$ and $A_{\rm 2}$ are positive constants. 
The functions $P_{\infty 1}$ and $P_{\infty 2}$ are given in \eqref{eq:power_flows}, and thus it is possible to choose an open set $D_{\rm 1} \subset X$ such that ${\rm det}{\sf A}(x)\neq0$. 
Thus, the model has vector relative degree $\{5,4\}$ at $x \in {D}_1$. 
Then, in terms of the coordinate transformation \eqref{eq:coordinate_transformation}, the determinant of the Jacobian matrix is given by 
\begin{align}
\label{eq:transformation}
{\rm det}{\rm D}\Phi(x) =& -F_0 \left\{\dfrac{{\rm d}P_{\infty2}}{{\rm d}x_{\rm e3}} (x_{\rm e3}) \right\}^3  \notag \\ &
  \cdot\left\{ F_1\dfrac{{\rm d}P_{\infty1}}{{\rm d}x_{\rm e1}}(x_{\rm e1})+F_2\dfrac{{\rm d}P_{\infty2}}{{\rm d}x_{\rm e3}}(x_{\rm e3}) \right\} \notag \\
 & \cdot\left\{ F_3\dfrac{{\rm d}P_{\infty1}}{{\rm d}x_{\rm e1}}(x_{\rm e1})+F_4\dfrac{{\rm d}P_{\infty2}}{{\rm d}x_{\rm e3}}(x_{\rm e3}) \right\}
\end{align}
where $F_0, \dots, F_4$ are positive constants determined by the parameters of the model. 
Since the function \eqref{eq:transformation} does not vanish for all $x$, there exists an open set $D_2 \subset X$ such that ${\rm det}{\rm D}\Phi(x)\neq 0$ at $x\in D_2$. 
From \eqref{eq:decoupling} and \eqref{eq:transformation},  it is possible to choose an open set $D$ satisfying $D:=D_1 \cap D_2$, and the proof is thus completed. 

\section{Proof of Lemma \ref{lem:independence_eta4}} \label{sec:independence_eta4}

The function $q(\xi_{\rm e},\,\xi_{\rm h},\,\eta)$, which describes the internal dynamics, is given by $r(x)$ in \eqref{eq:function_r}  through the coordinate transformation $\Phi$. 
Thus, it needs to be clarified how the state variables $x_{\rm g}$, $x_{\rm e}$, and $x_{\rm h}$ depend on the variables $\xi_{\rm e}$, $\xi_{\rm h}$, and $\eta$.  
Particularly, if the variables appearing in $r$ do not depend on $\eta_4$, the proof is completed. 
To this end, we investigate how the variables $\xi_{\rm e}$ and $\xi_{\rm h}$ depend on $x$. 
Because the dynamics of the heat subsystem do not affected by the averaged pressure level, $\xi_{\rm e}$ and $\xi_{\rm h}$ are independent of the variable. 
Thus, we can define the following twelve-dimensional coordinate transformation $\Psi$: 
\begin{align}
 \Psi :  (x_{\rm g} ,\,x_{\rm e}, x_{\rm h1}, \, x_{\rm h2})  
\mapsto (\xi_{\rm e},\,\xi_{\rm h}, [\eta_1,\, \eta_2,\, \eta_3 ]^\top).  
\end{align}
Here ${\rm det}{\rm D}\Psi\neq 0$ holds on the subset $D$ stated in Lemma\,\ref{thm:normal_form} because ${\rm det}{\rm D} \Phi \neq 0$. 
Thus, the coordinate transformation $\Psi$ is one-to-one on the subset $\hat{D}$, and the variables $x_{\rm e}$ and $x_{\rm g}$ as well as $x_{\rm h1}$ and $x_{\rm h2}$ are independent of $\eta_4$. 
The proof is thus completed.

\section{Proof of Lemma~\ref{lem:redefined_normal_form}} \label{sec:proof_redefined_normal_form}

The proof is completed in the same way as that of Lemma\,\ref{thm:normal_form}. 
For the redefined output, the new decoupling matrix $\hat{\sf A}(x)$ is obtained by differentiating the outputs $y_1$ and $\hat{y}_2$ in 5 and 3 times, respectively: 
\begin{align}
 \hat{\sf A}(x) &:= 	
\begin{bmatrix} L_{g_1}L_f^4 h_{\rm e}(x) &L_{g_2}L_f^4 h_{\rm e}(x)  \\[1.5ex]
 L_{g_1} L_f^2 \hat{h}_{\rm h}(x) & L_{g_2} L_f^2 \hat{h}_{\rm h}(x) \end{bmatrix}.
\end{align}
The matrix $\hat{\sf A}(x)$ depends only on $x_{\rm e}$, and its determinant is given by
\begin{align}
\label{eq:decoupling_hat}
{\rm det}\hat{\sf A}(x)= -\hat{A}_1 \dfrac{{\rm d}P_{\rm \infty1}}{{\rm d}x_{\rm e1}}(x_{\rm e1}) +\hat{A}_2 \dfrac{{\rm d}P_{\rm 2\infty}}{x_{\rm e3}}(x_{\rm e3}),  
\end{align}
where $\hat{A}_1$ and $\hat{A}_2$ are positive constants.  
Thus, it is possible to choose an open set $\hat{D}_1 \subset X$ such that ${\rm det}\hat{\sf A}(x)\neq0$, and the model has vector relative degree $\{5,\,3\}$ at $x\in \hat{D}_1$. 
Then, in terms of the coordinate transformation \eqref{eq:coordinate_transformation_redefined}, the determinant of the Jacobian matrix is given by 
\begin{align} \label{eq:transformation_hat}
{\rm det}{\rm D}\hat{\Phi}(x) =& \hat{F}_0 \left \{ \dfrac{{\rm d}P_{2\infty}}{{\rm d}x_{\rm e3}}(x_{\rm e3})\right\}^3 \notag \\ & 
\cdot \left\{ \hat{F}_1 \dfrac{{\rm d}P_{1\infty}}{{\rm d}x_{\rm e1}}(x_{\rm e1})-\hat{F_2}\dfrac{{\rm d}P_{2\infty}}{{\rm d}x_{\rm e3}}(x_{\rm e3})\right\} \notag \\
 & \cdot \left\{ \hat{F}_3 \dfrac{{\rm d}P_{1\infty}}{{\rm d}x_{\rm e1}}(x_{\rm e1})-\hat{F_4}\dfrac{{\rm d}P_{2\infty}}{{\rm d}x_{\rm e3}}(x_{\rm e3}) \right\} 
\end{align}
where $\hat{F}_0, \dots, \hat{F}_4$ are positive constants.  
Since the function \eqref{eq:transformation_hat} does not vanish for all $x$, there exists an open set $\hat{D}_2 \subset X$ such that ${\rm det}{\rm D}\hat{\Phi}(x)\neq 0$ at $x\in \hat{D}_2$. 
From \eqref{eq:decoupling_hat} and \eqref{eq:transformation_hat},  it is possible to choose an open set $\hat{D}$ satisfying $\hat{D}:=\hat{D}_1 \cap \hat{D}_2$, and the proof is thus completed.

\section{Proof of Lemma~\ref{th:internal_dynamics}} \label{sec:proof_internal_dynamics}

The proof is completed in a similar way to that in Appendix~\ref{sec:independence_eta4}.  
We first investigate how the variables $\xi_{\rm e}$ and $\hat{\xi}_{\rm h}$ depend on $x$. 
From \eqref{eq:system_equation}, it follows that the variable $\xi_{\rm e}$ depends only on $x_{\rm e}$ and $x_{\rm g}$. 
It is also verified that $\hat{\xi}_{\rm h}$ depends only on $x_{\rm h3}$ and $x_{\rm g}$. 
Particularly, the variables $\hat{\xi}_{\rm h2}$ and $\hat{\xi}_{\rm h3}$ depend only on $x_{\rm g}$ because  the dynamics of the heat subsystem are not affected by the averaged pressure level.
Thus, we can define the following ten-dimensional coordinate transformation $\hat{\Psi}$: 
\begin{align}
 \hat{\Psi}:  (x_{\rm g},\,x_{\rm e},\,x_{\rm h1},\, x_{\rm h2} ) 
 \mapsto (\xi_{\rm e},\,[\hat{\xi}_{\rm h2},\,\hat{\xi}_{\rm h3}]^\top, \hat{\eta}).  
\end{align}
Here ${\rm det}{\rm D}\hat{\Psi}\neq 0$ holds on the subset $\hat{D}$ stated in Lemma\,\ref{lem:redefined_normal_form} because ${\rm det}{\rm D} \Phi \neq 0$. 
Thus, the coordinate transformation $\hat{\Psi}$ is one-to-one on the subset $\hat{D}$, and the variables $x_{\rm e}$ and $x_{\rm g}$ as well as $x_{\rm h1}$ and $x_{\rm h2}$ are independent of $\hat{\xi}_{\rm h1}$. 
The proof is thus completed.

\end{document}